\documentclass[oneside,english]{amsart}
\usepackage[T1]{fontenc}
\usepackage[latin9]{inputenc}
\usepackage{geometry}
\usepackage{color}
\usepackage{babel}
\usepackage{verbatim}
\usepackage{mathrsfs}
\usepackage{enumitem}
\usepackage{amstext}
\usepackage{amsthm}
\usepackage{amssymb}
\usepackage{stmaryrd}
\usepackage[all]{xy}
\usepackage[unicode=true,pdfusetitle,
 bookmarks=true,bookmarksnumbered=false,bookmarksopen=false,
 breaklinks=false,pdfborder={0 0 1},backref=false,colorlinks=true]
 {hyperref}

\makeatletter

\newcommand{\ZZ}{{\mathbb Z}}
\newcommand{\NN}{{\mathbb N}}
\newcommand{\QQ}{{\mathbb Q}}

\newcommand{\calO}{{\mathcal O}}
\newcommand{\calE}{{\mathcal E}}

\newcommand{\calF}{{\mathcal F}}

\DeclareMathOperator{\Spec}{Spec}

\newcommand{\FF}{\mathbb{F}}
\newcommand{\Fq}{\mathbb{F}_q}

\newcommand{\Aut}{\textrm{Aut}}
\newcommand{\Hom}{\textrm{Hom}}

\newcommand{\Zp}{\mathbb{Z}_p}
\newcommand{\Zl}{\mathbb{Z}_l}
\newcommand{\Qp}{\mathbb{Q}_p}
\newcommand{\Qpbar}{\overline{\mathbb Q}_p}
\newcommand{\Ql}{\mathbb{Q}_l}
\newcommand{\Qlbar}{\overline{\mathbb{Q}}_l}

\newcommand{\Qlambar}{\overline{\mathbb Q}_{\lambda}}

\newcommand{\fisoc}[1]{\textbf{F-Isoc}(#1)}
\newcommand{\fisocd}[1]{\textbf{F-Isoc}^{\dagger}(#1)}
\newcommand{\fc}[1]{\textbf{FC}(#1)}

\numberwithin{equation}{section}
\numberwithin{figure}{section}

\theoremstyle{plain}
\newtheorem{thm}{\protect\theoremname}[section]
\theoremstyle{empty}

\theoremstyle{plain}

\theoremstyle{remark}
\newtheorem{rem}[thm]{\protect\remarkname}
\theoremstyle{plain}
\newtheorem*{thm*}{\protect\theoremname}
\theoremstyle{plain}
\newtheorem*{cor*}{\protect\corollaryname}
\theoremstyle{remark}
\newtheorem*{acknowledgement*}{\protect\acknowledgementname}
\theoremstyle{definition}
\newtheorem{defn}[thm]{\protect\definitionname}
\theoremstyle{plain}

\theoremstyle{remark}

\theoremstyle{definition}

\theoremstyle{plain}
\newtheorem{cor}[thm]{\protect\corollaryname}
\theoremstyle{remark}
\newtheorem*{rem*}{\protect\remarkname}
\theoremstyle{plain}
\newtheorem{lem}[thm]{\protect\lemmaname}
\newlist{casenv}{enumerate}{4}
\setlist[casenv]{leftmargin=*,align=left,widest={iiii}}
\setlist[casenv,1]{label={{\itshape\ \casename} \arabic*.},ref=\arabic*}
\setlist[casenv,2]{label={{\itshape\ \casename} \roman*.},ref=\roman*}
\setlist[casenv,3]{label={{\itshape\ \casename\ \alph*.}},ref=\alph*}
\setlist[casenv,4]{label={{\itshape\ \casename} \arabic*.},ref=\arabic*}

\theoremstyle{plain}
\newtheorem{prop}[thm]{\protect\propositionname}
\theoremstyle{plain}

\usepackage[all]{xy}
\usepackage{url}
\theoremstyle{definition}
\newtheorem{setup}[thm]{Setup}

\makeatother

\providecommand{\acknowledgementname}{Acknowledgement}
\providecommand{\casename}{Case}
\providecommand{\conjecturename}{Conjecture}
\providecommand{\corollaryname}{Corollary}
\providecommand{\definitionname}{Definition}
\providecommand{\examplename}{Example}
\providecommand{\factname}{Fact}
\providecommand{\lemmaname}{Lemma}
\providecommand{\notationname}{Notation}
\providecommand{\propositionname}{Proposition}
\providecommand{\questionname}{Question}
\providecommand{\remarkname}{Remark}
\providecommand{\theoremname}{Theorem}

\makeatletter
\@namedef{subjclassname@2020}{\textup{2020} Mathematics Subject Classification}
\makeatother

\begin{document}

\title{Rank 2 Local Systems and Abelian Varieties II}

\author{Raju Krishnamoorthy}
\email{krishnamoorthy@alum.mit.edu }  
\address{Bergische Universit\"at Wuppertal,
F13.05,
Gau{\ss}stra{\ss}e 20, Wuppertal}
\author{Ambrus P\'al}
\email{a.pal@imperial.ac.uk}
\address{Department of Mathematics,
180 Queens Gate, Imperial College, London, SW7 2AZ,
United Kingdom}
\subjclass[2020]{11G10, 14D10, 14F30, 14G35}

\begin{abstract}
Let $X/\mathbb{F}_{q}$ be a smooth, geometrically connected, quasi-projective scheme. Let $\mathcal{E}$ be a semisimple overconvergent $F$-isocrystal on $X$. Suppose that irreducible summands $\mathcal{E}_i$ of $\mathcal E$ have rank 2, determinant $\bar{\mathbb{Q}}_p(-1)$, and infinite monodromy at $\infty$. Suppose further that for each closed point $x$ of $X$, the characteristic polynomial of $\mathcal{E}$ at $x$ is in $\mathbb{Q}[t]\subset \mathbb Q_p[t]$. Then there exists a dense open subset $U\subset X$ such that $\mathcal{E}|_U$ comes from a family of abelian varieties on $U$.

As an application, let $L_1$ be an irreducible lisse $\bar{\mathbb{Q}}_l$ sheaf on $X$ that has rank 2, determinant $\bar{\mathbb{Q}}_l(-1)$, and infinite monodromy at $\infty$. Then all crystalline companions to $L_1$ exist (as predicted by Deligne's crystalline companions conjecture) if and only if there exists a dense open subset $U\subset X$ and an abelian scheme $\pi_U\colon A_U\rightarrow U$ such that $L_1|_U$ is a summand of $R^1(\pi_U)_*\bar{\mathbb{Q}}_l$.

\end{abstract}

\maketitle


\section{Introduction}
Throughout this article, $p$ is a prime number and $q$ is a power of $p$. If $X/k$ is a smooth scheme over a perfect field of characteristic $p$, then $\fisocd{X}$ denotes the category of overconvergent $F$-isocrystals on $X$ and $\fisocd{X}_{\Qpbar}$ denotes its $\Qpbar$-linearization. Overconvergent $F$-isocrystals are a $p$-adic analog of lisse $l$-adic sheaves.
\begin{defn}Let $X/k$ be a smooth, geometrically connected scheme over a perfect field $k$ of characteristic $p$ and let $\calE\in\fisocd{X}_{\Qpbar}$. We say that $\calE$ has \emph{infinite local monodromy at infinity} if for every triple $(X',\overline{X'},f)$ where $\overline{X'}$ is smooth projective over $k$, $X'\subset \overline{X'}$ is a dense Zariski open subset, and $f\colon X'\rightarrow X$ is an alteration, the overconvergent $F$-isocrystal $f^*\calE$ does not extend to an $F$-isocrystal on $\overline{X'}$.
\end{defn}
This definition of infinite local monodromy at infinity applies equally well to lisse $\Qlbar$-sheaves and is compatible with the other notions of infinite local monodromy at infinity.
\begin{thm}\label{main_thm}Let $X/\Fq$ be a smooth, geometrically connected, quasi-projective scheme. Let $\calE\in \fisocd{X}$ be a semisimple overconvergent $F$-isocrystal. Suppose:
\begin{itemize}
\item for every closed point $x$ of $X$, the polynomial $P_x(\calE,t)$ has coefficients in $\QQ\subset \Qp$;
\item every irreducible summand $\calE_i\in \fisocd{X}_{\Qpbar}$ of $\calE$ has rank 2, determinant $\Qpbar(-1)$, and infinite local monodromy around infinity.
\end{itemize}
Then $\calE$ comes from a family of abelian varieties. More precisely, there exists a non-empty open subset $U\subset X$ and an abelian scheme $A_U\rightarrow U$, so that $\mathbb D (A_U[p^{\infty}])\otimes \Qpbar\cong \mathcal E|_U$.
\end{thm}
Here, if $G\rightarrow X$ is a $p$-divisible group, $\mathbb{D}(G)$ is the (contravariant) Dieudonn\'e crystal attached to $G$. We have the following applications. Deligne formulated what is now called the companions conjecture in \cite[Conjecture 1.2.10 (vi)]{deligne1980conjecture}. For a guide to the crystalline companions conjecture, see \cite{kedlaya2016notes, kedlayacompanions}.
\begin{cor}\label{cor:higher_dimension_drinfeld}Let $X/\Fq$ be a smooth, geometrically connected, quasi-projective scheme. Let $L_1$ be an irreducible rank 2 lisse $\Qlbar$ sheaf on $X$ with infinite monodromy around infinity and determinant $\Qlbar(-1)$. Then the following are equivalent:
\begin{enumerate}
\item there exists a non-empty open subset $U\subset X$ and an abelian scheme $\pi\colon A_U\rightarrow U$ such that $L_1|_U$ is a summand of $R^1(\pi_U)_*\Qlbar$; 
\item all crystalline companions to $L_1$ exist (as predicted by Deligne's crystalline companions conjecture).
\end{enumerate}
\end{cor}
\begin{cor}\label{cor:higher_dimensional_padic_drinfeld}
Let $X/\Fq$ be a smooth, geometrically connected, quasi-projective scheme. Let $\calE_1$ be an irreducible rank 2 object of $\fisocd{X}_{\Qpbar}$ with infinite monodromy around infinity and determinant $\Qpbar(-1)$. Suppose the (number) field $E_1\subset \Qpbar$ generated by the coefficients of $P_x(\calE_1,t)$ as $x$ ranges through the closed points of $X$ has a single prime over $p$. Then $\calE_1$ comes from a family of abelian varieties: there exists a non-empty open subset $U\subset X$ and an abelian scheme $A_U\rightarrow U$ such that $\calE_1|_U$ is a summand of $\mathbb{D}(A_U[p^{\infty}])\otimes\Qpbar$.
\end{cor}

In particular, Corollaries \ref{cor:higher_dimension_drinfeld} and \ref{cor:higher_dimensional_padic_drinfeld} provide some evidence for a question of Drinfeld \cite[Question 1.4]{drinfeld2012conjecture} and a conjecture of the authors \cite[Conjecture 1.2]{krishnamoorthypal2018}. Our motivation for formulating this conjecture was a celebrated theorem of Corlette-Simpson over $\mathbb{C}$ \cite[Theorem 11.2]{corlette2008classification}, the proof of which uses non-abelian Hodge theory. In contrast to our earlier work \cite{krishnamoorthypal2018}, this article does not use Serre-Tate deformation theory nor does it use the algebraization/globalization techniques of \cite{hartshorneample}.

We briefly sketch the proof. Drinfeld's first work on the Langlands correspondence for $GL_2$, together with Abe's work on the $p$-adic Langlands correspondence and Lemma \ref{lem:all_companions}, implies Theorem \ref{main_thm} when $\dim(X)=1$. (The precise argument is given in Step 2 and also uses  Remark \ref{rem:mult_1} to organize the summands, as explained in Step 1 of the proof.) Note that the resulting abelian scheme is not unique, but it is unique up to isogeny.

To do the higher-dimensional case, we first assume that $X$ admits a simple normal crossings compactification $\bar{X}$ and $\calE$ is a logarithmic $F$-isocrystal with nilpotent residues. (We recall the notion of logarithmic $F$-isocrystals in Appendix \ref{section:log}.) A technique of Katz, combined with slope bounds originally due to Lafforgue, allow one to construct a (non-canonical) logarithmic Dieudonn\'e crystal on an open set $U$ of the compactification $\bar X$ whose associated logarithmic $F$-isocrystal is isomorphic to the restriction $\calE|_U$. After the work of Kato-Trihan, this logarithmic Dieudonn\'e crystal yields a natural line bundle, which we call the \emph{Hodge bundle} $\omega$ of the logarithmic Dieudonn\'e crystal, on $\bar{X}$.

For any odd prime $l\neq p$, let $\mathscr{A}_{h,1,l}$ denote the moduli space of principally polarized abelian schemes of dimension $h$ equipped with full level $l$ structure over $\Spec(\ZZ[1/l])$. It is well-known that the Hodge line bundle is ample on $\mathscr{A}_{h,1,l}$ over $\Spec(\ZZ[1/l])$. We use Poonen's Bertini theorem over finite fields together with Drinfeld's result and Zarhin's trick to find a well-adapted family of extremely ample space-filling curves $\bar{C}_n$ of $\bar{X}$ that each map to the minimal compactification $\mathscr{A}^*_{h,1,l}\subset \mathbb{P}^m$ via some \emph{fixed} power of the Hodge bundle $\omega|_{\bar C_n}^r$. (This step uses foundational work of \'Etesse, Kato, Kedlaya,  and Trihan that we explain in Appendix \ref{section:log}.) Note that $H^0(\bar{X},\omega^r)$ is a finite dimensional vector space over a finite field and is hence a finite set. We use this finitude together with the pigeonhole principle to prove that infinitely many of these maps can be pieced together to a rational map $\bar{X}\dashrightarrow \mathscr{A}_{h,1,l}\subset \mathbb{P}^m$. Therefore we obtain an abelian scheme $\psi_U\colon B_U\rightarrow U$ over some open $U\subset X$. The space-filling properties of the $\bar{C}_n$ and Zarhin's work on the Tate isogeny theorem for fields finitely generated over $\Fq$ then allow us to conclude. 

To deduce the general case, we use Kedlaya's semistable reduction theorem for overconvergent $F$-isocrystals.
\begin{rem}We comment on the relation of this article to \cite{krishnamoorthypal2018}. In \cite{krishnamoorthypal2018}, we prove a Lefschetz-style theorem for families of $\text{GL}_2$ type abelian schemes over finite fields. This has the following implication for \cite[Conjecture 1.2]{krishnamoorthypal2018}: if $X/\Fq$ is a smooth projective variety, then there exists an ample curve $C\subset X$ such that if $\calE\in\fisoc{X}_{\Qpbar}$ and $\calE|_C$ comes from an abelian scheme $A_C\rightarrow C$ of $\text{GL}_2$-type, then there is an open subset $U\subset X$ such that $\calE|_U$ comes from an abelian scheme $B_U\rightarrow U$ of $\text{GL}_2$-type. (It follows from Zarhin's work on the Tate isogeny conjecture that $B_C\rightarrow C$ is indeed isogenous to $A_C\rightarrow C$.) To prove this, we use Serre-Tate deformation theory and globalization results of \cite{hartshorneample}, the latter of which critically uses the positivity of $C$ in $X$.  In this article, we only deal with \emph{non-proper} varieties $X/\FF_q$ and we use infinitely many (space-filling, affine) curves together with a result of Drinfeld, which is only known for affine curves. In particular, the main results of \cite{krishnamoorthypal2018} do not imply the main result of this article.
\end{rem}

\section{Preliminaries}
Before proving Theorem \ref{main_thm}, we need several preliminary results. A key ingredient in the proof is the following, which is a byproduct of Drinfeld's first work on the Langlands correspondence for $GL_2$.
\begin{thm}
\label{Theorem:GL2}(Drinfeld) Let $C/\Fq$ be a smooth affine curve and let $L_1$ be a rank 2 irreducible $\Qlbar$ sheaf with determinant $\Qlbar(-1)$. Suppose $L_1$ has infinite local monodromy around some point at $\infty\in\overline{C}\backslash C$. Then $L_1$ comes from a family of abelian varieties in the following sense: let $E$ be the field generated by the Frobenius traces of $L_1$ and suppose $[E:\QQ]=g$. Then there exists an abelian scheme
\[
\pi_C\colon A_{C}\rightarrow C
\]
of dimension $g$ and an isomorphism $E\cong \textrm{End}_{C}(A)\otimes\QQ$, realizing $A_C$ as a $GL_{2}$-type abelian scheme, such that $L_1$ occurs as a summand of $R^1(\pi_C)_*\Qlbar$. Moreover, $A_{C}\rightarrow C$ is totally degenerate around $\infty$.
\end{thm}

See \cite[Proof of Proposition 19, Remark 20]{snowden2018constructing} for how to recover this result from Drinfeld's work. This amounts to combining \cite[Main Theorem, Remark 5]{drinfeld1983} with \cite[Theorem 1]{drinfeld1977}.

For completeness, we briefly recall the theory of companions and what is known about them. For a thorough summary about the definitions and also what is known, we refer the reader to \cite{kedlayacompanions}. Alternatively, the reader may see \cite[Section 4]{krishnamoorthypal2018}.

\begin{defn}\label{def:companions}Let $X/\mathbb F_q$ be a smooth, geometrically connected variety. Let $\lambda$ be a prime number and let $\mathcal E$ denote either a smooth $\Qlambar$ sheaf on $X$ if $\lambda\neq p$ or an overconvergent $F$-isocrystal with coefficients in $\Qpbar$ if $\lambda=p$. Following Kedlaya  \cite[Section 1]{kedlayacompanions}, we call such $\calE$ \emph{coefficient objects}.

\begin{enumerate}
\item Let $l\neq p$ be a prime number and let $L$ be a lisse $\Qlbar$-sheaf on $X$. Fix a (possibly non-continuous) field isomorphism $\iota\colon \Qlambar\rightarrow \Qlbar$. We say that $L$ is an \emph{$\iota$-companion of $\mathcal E$} if for all closed points $x\in X$, we have:
$$\iota(P_x(\mathcal E,t))=P_x(L,t)\in \Qlbar[t],$$
where $P_x(-,t)$ denotes the reverse characteristic polynomial at the closed point $x$. 
\item Let $\mathcal F$ be an overconvergent $F$-isocrystal on $X$ with coefficients in $\Qpbar$ and fix an isomorphism $\iota\colon \Qlambar\rightarrow \Qpbar$. We say $\mathcal F$ is an \emph{$\iota$-companion of $\mathcal E$} if for all closed points $x\in X$, we have:
$$\iota(P_x(\mathcal E,t))=P_x(\mathcal F,t)\in \Qpbar[t].$$

\end{enumerate}
In either of these cases, we say that the $\iota$-companion to $\mathcal E$ exists.

\end{defn}
Suppose $\mathcal E$ is semi-simple and each irreducible summand has algebraic determinant. Then Deligne's conjecture, together with Crew's $p$-adic enchancement, predict that all $\iota$-companions to $\calE$ exist. It follows from work of Abe, Abe-Esnault, Deligne, Drinfeld, Kedlaya, and Lafforgue \cite{abe2016lefschetz, abe2013langlands,deligne2012finitude,drinfeld2012conjecture, lafforgue2002chtoucas} this conjecture is known to hold in the following cases.
\begin{thm}\label{Theorem:companions}Let $X/\FF_q$ be a smooth, geometrically connected variety. Let $\mathcal E$ be a semi-simple coefficient object on $X$ such that the irreducible summands have algebraic determinant. Then:
\begin{itemize}
\item If $\dim(X)=1$, then all $\iota$-companions exist (\cite[Th\'eor\`eme VII.6]{lafforgue2002chtoucas} and \cite[Theorem 4.4.1]{abe2013langlands}).
\item For any $l\neq p$ and any isomorphism $\iota\colon \Qlambar\rightarrow \Qlbar$, the $\iota$-companion to $\mathcal E$ exists (\cite[Theorem 1.1]{drinfeld2012conjecture} and \cite[Theorem 4.2]{abe2016lefschetz} or \cite[Theorem 0.4.1]{kedlayacompanions}).
\end{itemize}
\end{thm}

In particular, $p$-adic companions are not known to exist when $\dim(X)>1$, although Kedlaya has recently proposed a promising strategy \cite{kedlayacompanionsii}.

\begin{prop}\label{prop:shape_ell_companion}Maintain the hypotheses of Theorem \ref{main_thm}. Let $\iota\colon \Qpbar\rightarrow \Qlbar$ be a field isomorphism and let $L:=\ ^{\iota}\calE$ be the (semi-simple) $\iota$-companion to $\calE$. Then:
\begin{itemize}
\item The isomorphism class of $L$ is independent of the choice of $\iota$.
\item Let $L_i$ be an irreducible summand of $L$. Then $L_i$ has rank 2, determinant $\Qlbar(-1)$, and infinite monodromy at infinity. 
\end{itemize}
\end{prop}
\begin{proof}
For all closed points $x$ of $X$, we have that $P_x(L,t)\in \QQ[t]\subset \Qlbar[t]$ as $\iota(\QQ)=\QQ\subset\Qlbar$. The first statement then follows from the Cebotarev density theorem and the Brauer-Nesbitt theorem.

If $\calE_i$ is an irreducible summand of $\calE$, then $^{\iota}\calE_i$ is an irreducible $\Qlbar$-sheaf by \cite[Theorem 3.3.1]{kedlayacompanions}. As the companions relation commutes with direct sum, it follows that if $\calE\cong \oplus \calE_i^{m_i}$ is the decomposition of $\calE$ into irreducible objects in $\fisocd{X}_{\Qpbar}$, then $L\cong\oplus (^{\iota}\calE_i)^{m_i}$ is a decomposition of $L$ into irreducible lisse $\Qlbar$-sheaves on $X$. One may observe that $\det(^{\iota}\calE_i)\cong \Qlbar(-1)$ because for every closed point $x$ of $X$, the constant term of $P_x(\calE_i,t)$ is $q$ and hence the constant term of $P_x(^{\iota}\calE_i,t)$ is also $q$. Finally, suppose for contradiction that there exists an $i$ with $L_i:=\ ^{\iota}\calE_i$ having finite local monodromy at infinity. Then there exists a smooth projective variety $\bar{X}'/\Fq$, and open dense subscheme $X'\subset \bar{X}'$, and an alteration $f\colon X'\rightarrow X$ such that $f^*L_i$ extends to $\bar{X}'$. It follows from \cite[Corollary 3.3.3]{kedlayacompanions} that $f^*\calE_i$ also extends to $\bar{X}'$, contradicting the hypothesis that $\calE_i$ had infinite local monodromy at infinity.
\end{proof}

We will need the following lemma to ensure that, given the hypotheses of Theorem \ref{main_thm}, every $p$-adic companion of $\calE_i$ is again a summand of $\calE$; moreover, the companion relation preserves multiplicity in the isotypic decomposition of $\calE$. 
\begin{lem}\label{lem:all_companions}
Let $X/\Fq$ be a smooth, geometrically connected scheme.

\begin{enumerate}

\item Let $l\neq p$ be a prime and let $L$ be a lisse, semi-simple  $\Qlbar$-sheaf on $X$, all of whose irreducible summands $L_i$ have algebraic determinant. Suppose for all closed points $x$ of $X$, we have:
$$P_x(L,t)\in \QQ[t]\subset \Qlbar[t].$$
Let $L_i$ be an irreducible summand of $L$ that occurs with multiplicity $m_i$ and $\iota\in \text{Aut}_{\QQ}(\Qlbar)$ be a field automorphism. Then the $\iota$-companion to $L_i$, denoted $^{\iota}L_i$, is isomorphic to an irreducible summand of $L$ that occurs with multiplicity $m_i$.

 \item Let $\calF$ be a semi-simple object of $\fisocd{X}_{\Qpbar}$, all of whose irreducible summands $\calF_i$ have algebraic determinant. Suppose for all closed points $x$ of $X$, we have:
$$P_x(\calF,t)\in \QQ[t]\subset \Qpbar[t].$$
Let $\calF_i$ be an irreducible summand of $\calF$ that occurs with multiplicity $m_i$. Let $\iota\in \text{Aut}_{\QQ}(\Qpbar)$. Then the $\iota$-companion of $\calF_i$, denoted $^{\iota}\calF_i$,exists and is isomorphic to a direct summand of $\calF$  that occurs with multiplicity $m_i$.
\end{enumerate}
\end{lem}
\begin{proof}
We reduce the crystalline case to the \'etale case. (Note that we could have equivalently proceeded by reduction to curves using \cite{abe2016lefschetz}.) As $\calF$ is semisimple, write an isotypic decomposition:
$$\displaystyle \calF\cong \bigoplus_{i=1}^a\calF_i^{m_i}.$$
Note that each $\calF_i$ is pure by \cite[Theorem 2.7]{abe2016lefschetz}. Fix an isomorphism $\sigma\colon \Qpbar\rightarrow \Qlbar$.
By \cite[Theorem 4.2]{abe2016lefschetz} or \cite[Corollary 3.5.3]{kedlayacompanions}, the $\sigma$-companion to each $\calF_i$ exists as an irreducible lisse $\Qlbar$-sheaf $L_i$. Setting $L$ to be the semi-simple $\sigma$-companion of $\calF$, we have:
$$\displaystyle L\cong \bigoplus_{i=1}^aL_i^{m_i}.$$

Set $\iota\in \text{Aut}_{\QQ}(\Qpbar)$. Then $\calF_j$ is the $\iota$-companion to $\calF_i$ if and only if $L_j$ is the $\sigma\circ\iota\circ \sigma^{-1}$-companion to $L_i$. Therefore it suffices to prove the result in the \'etale setting.

Let $M$ be an irreducible lisse $\Qlbar$-sheaf on $X$. Then $M$ is pure by \cite[Th\'eor\`eme 1.6]{deligne2012finitude} and class field theory. Then the multiplicity of $M$ in the semisimple sheaf $L$ is: $\dim(H^0(X,M^{*}\otimes L))$. By assumption we have that for all closed points $x$ of $X$, $P_x(L,t)\in \QQ[t]\subset \Qlbar[t]$. Let $\iota\in \Aut_{\QQ}(\Qlbar)$, and note that the semi-simple $\iota$-companion to $L$ is again isomorphic to $L$. Then we claim that the $\iota$-companion to $M^*\otimes L$ is isomorphic to $(^\iota M^*)\otimes L$. Indeed, this follows from the following two facts. First of all, both $M^*\otimes L$ and $(^\iota M^*)\otimes L$, being the tensor product of semi-simple representations of characteristic 0, are semi-simple. Second of all, it follows from the fundamental theorem of symmetric functions that for fixed $d,e\in \mathbb N$ there exist universal polynomials $(u_i)_{i=0}^{de}$ in the ring $\mathbb Q[\alpha_1,\dots,\alpha_d,\beta_1,\dots,\beta_e]$ with the following property: let $V$ and $W$ be finite dimensional vector spaces over a field $K$ of characteristic 0 and of dimensions $d$ and $e$ and let $A$ and $B$ be linear operators on $V$ and $W$ respectively. Write $P(A,t)=\sum_{i=0}^{d}a_it^i$ and $P(B,t)=\sum_{j=0}^{e}b_jt^j$ for the reverse characteristic polynomials of $A$ and $B$. Then the reverse characteristic polynomial $P(A\otimes B,t)$ of $A\otimes B$ is equal to:
$$\displaystyle P(A\otimes B,t)=\sum_{k=0}^{de}u_k(a_0,\dots,a_d,b_0,\dots,b_e)t^k.$$

Translating back, let $x$ be a closed point of $X$ and write $P_x(M^*,t)=\sum_{i=0}^d a_it^i$ and $P_x(L,t)=\sum_{j=0}^e b_jt^j$. Then we have:
$$\displaystyle P_x(M^*\otimes L,t)=\sum_{k=0}^{de}u_k(a_0,\dots,a_d,b_0,\dots,b_e)t^k.$$
It follows that
$$P_x(^{\iota}(M^*\otimes L),t)=P_x((^{\iota}M^*)\otimes\ ^{\iota}L,t)=\sum_{k=0}^{de}u_k(\iota(a_0),\dots,\iota(a_d),\iota(b_0),\dots,\iota(b_e))t^k.$$ 
But $\iota(b_j)=b_j$ because $b_j\in \mathbb Q$ for all $j$. Therefore $P_x(^{\iota}(M^*\otimes L),t)=P_x(^{\iota}(M^*)\otimes L,t)$. The semisimplicity of $^{\iota}(M^*\otimes L)$ and $^{\iota}M^*\otimes L$ allow us to conclude that $^{\iota}(M^*\otimes L)$ is isomorphic to $^{\iota}M^*\otimes L$.

On the other hand, the exact argument of \cite[3.2]{abe2016lefschetz} for lisse $l$-adic sheaves implies that $\dim(H^0(X,M^*\otimes L))=\dim(H^0(X,\ ^{\iota}(M^* \otimes L))$. Therefore $\dim(H^0(X,M^*\otimes L))=\dim(H^0(X,(^{\iota}M^*)\otimes L))$, and the result follows.

\end{proof}
\begin{rem}The argument of \cite[3.2]{abe2016lefschetz} cited in the proof of Lemma \ref{lem:all_companions} is based on \cite[Cor VI.3]{lafforgue2002chtoucas} and uses $L$-functions. A similar idea is used in the proof that the companions relations preserves irreducibility, which was crucial to Proposition \ref{prop:shape_ell_companion}. See also \cite[Lemma 3.1.5, Theorem 3.3.1]{kedlayacompanions}.
\end{rem}

\begin{rem}\label{rem:companions_tensor}It follows from the argument of Lemma \ref{lem:all_companions} that if $X/\Fq$ is smooth and geometrically connected and if $\calE, \calF\in \fisocd{X}_{\Qpbar}$ are semi-simple objects, all of whose summands are algebraic, then for any field isomorphism $\iota\colon \Qpbar\rightarrow \Qlbar$, we have: $^{\iota}(\calE\otimes \calF)\cong\ ^{\iota}\calE\otimes\ ^{\iota}\calF,$
i.e., the relation of being $\iota$-companions commutes with tensor product.
\end{rem}

\begin{rem}\label{rem:mult_1}It follows from Lemma \ref{lem:all_companions} that, in the context of Theorem \ref{Theorem:GL2}, there is a decomposition:

$$\displaystyle R^1(\pi_C)_*\Qlbar\cong \bigoplus_{i=1}^g(L_i)$$
where the $L_i$ form a complete set of $\Qlbar$ companions. There are exactly $g$ non-isomorphic companions because the field generated by Frobenius traces of $L_1$ is isomorphic to $E$ and the $l$-adic companions are in bijective correspondence with the embeddings $E\hookrightarrow \Qlbar$. In particular, each companion occurs with multiplicity 1. In fact, as $E\cong \text{End}_C(A_C)\otimes \QQ$, it follows that $E\otimes \Qlbar$ acts on $ R^1(\pi_C)_*\Qlbar$. On the other hand, $E\otimes \Qlbar\cong \prod_{i} \Qlbar$, where $i$ runs over the embeddings $E\hookrightarrow \Qlbar$. For each $i$, pick a non-trivial idempotent $e_i\in E\otimes \Qlbar$ whose image is the $i^\text{th}$ component of the direct product decomposition. The above direct sum decomposition is induced by these $e_i$.
\end{rem}

To apply Drinfeld's Theorem \ref{Theorem:GL2}, we will use the following lemma.

\begin{lem}\label{lem:map_on_curves}Let $Y/\Fq$ be a smooth, geometrically connected, projective scheme and let $\alpha$ be a line bundle on $Y$. Let $M\subset \mathbb{P}^m_{\Fq}$ be a closed subset. Suppose there exists an infinite collection $(C_n)_{n\in \mathbb{N}}$ of smooth, projective, geometrically connected, closed subcurves $C_n\subset Y$ such that
\begin{enumerate}
\item for each $n\in \mathbb{N}$, the natural map $H^0(Y,\alpha)\rightarrow H^0(C_n,\alpha|_{C_n})$ is an isomorphism;
\item for any infinite subset $S\subset \mathbb{N}$, the union:
$$\displaystyle \bigcup_{n\in S}C_n$$
is Zariski dense in $Y$; 
\item for each curve $C_n$, there exists $m+1$ globally generating sections 
$$t_{n,0},\dots,t_{n,m}\in H^0(C_n,\alpha|_{C_n})$$
such that the induced morphism to $\mathbb{P}^m$ factors through $M$:
\[
\xymatrix{C_n\ar[dr]\ar[r]^{f_n} & \mathbb{P}^m\\
 & M\ar[u].
}
\]
\end{enumerate}
Then there exist global sections $\tilde{t}_0,\dots,\tilde{t}_m\in H^0(Y,\alpha)$ such that the induced rational map $\tilde{f}\colon Y\dashrightarrow \mathbb{P}^m$ has image in $M$. Moreover, $\tilde{f}$ can be chosen to be compatible with infinitely many of the maps $f_n$. 
\end{lem}
\begin{proof}
There are finitely many ordered $m+1$-tuples of sections $H^0(Y,\alpha)\cong H^0(C_n,\alpha|_{C_n})$ because $H^0(Y,\alpha)$ is a finite dimensional vector space over $\Fq$. By the pigeonhole principle, in our infinite collection we may find an $m+1$-tuple of sections $\tilde{t}_0,\dots,\tilde{t}_m\in H^0(Y,\alpha)$ such that there exists an infinite set $S\subset \mathbb{N}$ with $$(\tilde{t}_0,\dots,\tilde{t}_m)|_{C_n}=(t_{n,0},\dots,t_{n,m})$$
for every $n\in S$. There is therefore an induced rational map $\tilde{f}\colon Y\dashrightarrow \mathbb{P}^m$ with $\tilde{f}|_{C_n}=f_n$ for each $n\in S$. On the other hand, the collection $(C_n)_{n\in S}$ is Zariski dense in $Y$ by assumption and $\tilde{f}(C_n)\subset M$; therefore the image of $\tilde{f}$ lands inside of $M$, as desired.
\end{proof}
Lemma \ref{lem:map_on_curves} has two key ingredients. The first ingredient is that if $X/\Fq$ is a projective variety and $\alpha$ is a coherent sheaf on $X$, then $H^0(X,\alpha)$ is a finite set. The second ingredient is the pigeonhole principle. To use  Lemma \ref{lem:map_on_curves}, the following definition will be useful.

\begin{defn}\label{def:good_curve}Let $\bar X/k$ be a smooth, geometrically connected, projective scheme of dimension at least 2, let $Z\subset \bar X$ be a reduced simple normal crossings divisor, and set $X:=X\backslash Z$. Let $\bar U\subset \bar X$ be an open subset whose complement has codimension at least 2. Let $(x_j)_{j=1}^s$  be a finite collection of closed points of $U:=\bar U\cap X$. Let $\alpha$ be a line bundle on $\bar X$. We say that $\bar C\subset \bar U$ is a \emph{good curve} for the quintuple $(\bar X, X, \bar U, \alpha, (x_j)^s_{j=1})$ if
\begin{itemize}
\item $\bar{C}$ is the smooth complete intersection of smooth ample divisors of $\bar{X}$ that intersect $Z$ in good position;
\item $\bar{C}$ contains each of the closed points $x_j$, for $j=1\dots s$;
\item the natural map $H^0(\bar{X},\alpha)\rightarrow H^0(\bar{C},\alpha|_{\bar C})$ is an isomorphism.
\end{itemize}
\end{defn}
In the proof of Theorem \ref{main_thm}, we will need to know that good curves exist. This is guaranteed by the following two results.
\begin{prop}\label{proposition:divisor_absorbs_sections}
Let $Y/k$ be a smooth, geometrically connected, projective scheme of dimension $d\geq 2$ and let $\alpha$ be a line bundle on $Y$. Let $D\subset Y$ be an ample divisor. Then there exists an $s_0>0$ such that for any $s\geq s_0$, and for any integral divisor $E\in |sD|$ in the linear series, the natural map:
$$H^0(Y,\alpha)\rightarrow H^0(E,\alpha|_E)$$
is an isomorphism.
\end{prop}
\begin{proof}
For any $s>0$, let $E\in |sD|$ be an integral divisor in the linear series. Then there is an exact sequence:
$$0\rightarrow \alpha(-E)\rightarrow \alpha\rightarrow \alpha|_E\rightarrow 0.$$
If $h^0(Y,\alpha(-E))=h^1(Y,\alpha(-E))=0$, then by the long exact sequence in cohomology, the restriction map $H^0(Y,\alpha)\rightarrow H^0(E,\alpha|_E)$ is an isomorphism. Our task is therefore to show that for all sufficiently large $s$, $h^0(Y,\alpha(-sD))=h^1(Y,\alpha(-sD))=0$.

Let $\mathfrak L$ be the canonical bundle of $Y$. Then by Serre duality, $h^i(Y,\alpha(-sD))=h^{d-i}(Y,\alpha^{\vee}(sD)\otimes \mathfrak L)$. It follows from Serre vanishing that there exists an $s_0>0$ such that for any $s\geq s_0$ and for any $i<d$, $h^{d-i}(Y,\alpha^{\vee}(sD)\otimes \mathfrak L)=0$. Therefore for any $s\geq s_0$ and for any $i<d$, $h^i(Y,\alpha(-sD))=0$ and the result follows.

\end{proof}

\begin{lem}\label{lem:good_curves_exist}Let $\bar X/\Fq$ be a smooth, geometrically connected, projective scheme of dimension at least 2, let $Z\subset \bar X$ be a reduced simple normal crossings divisor, and set $X:=\bar{X}\backslash Z$. Let $\bar U\subset \bar X$ be an open subset whose complement has codimension at least 2. Let $(x_j)_{j=1}^s$  be a finite collection of closed points of $U:=\bar U\cap X$. Let $\alpha$ be a line bundle on $\bar X$. Then there is a good curve $\bar C\subset \bar U$ for the quintuple $(\bar X, X, \bar U, \alpha, (x_j)^s_{j=1})$
\end{lem}
\begin{proof}
By induction, it suffices to construct a smooth ample divisor $\bar D\subset \bar X$ such that
\begin{itemize}
\item $\bar D\cap \bar U$ has complementary codimension at least 2 in $\bar D$;
\item $\bar D$ intersects $Z$ transversely;
\item $\bar D$ contains $x_j$, for $j=1\dots s$; and
\item the natural map $H^0(\bar{X},\alpha)\rightarrow H^0(\bar{D},\alpha|_{\bar D})$ is an isomorphism.
\end{itemize}

This is a standard application of Poonen's Bertini theorem over finite fields \cite[Theorem 1.3]{poonen2004bertini}. Fix a closed embedding $\bar X\hookrightarrow \mathbb P^m_{\Fq}$ and let $S_{\text{homog}}$ be the set of homogenous polynomials on $\mathbb P^m_{\Fq}$, as in  page 1 of \emph{loc. cit.} Consider the set $\mathcal T$ of those functions $f\in S_{\text{homog}}$ such that $\bar D:=V(f)\cap \bar X$ is a smooth ample divisor of $\bar X$ and the above four properties hold for $\bar D$. Our goal is to show that $\mathcal T$ is non-empty. 

\begin{itemize}
\item Let $\bar E:=\bar X\setminus \bar U$; by hypothesis, $\dim(\bar E)\leq n-2$. If $f\in S_{\text{homog}}$ is such that $V(f)$ does not contain any component of $\bar E$, then $\dim(V(f)\cap \bar E)\leq n-3$. For this to hold, it is sufficient that $V(f)$ avoids at least one given closed point $e_i$ on each connected component of $\bar E$. 
\item Write $\displaystyle Z=\cup^{r}_{j=1} Z_j$ to be the decomposition of $Z$ into connected components. For each $J\subset \{1,2,\dots,r\}$, set $\displaystyle Z_J:=\cap_{j\in J}Z_j$ to be the corresponding scheme-theoretic intersection. By assumption, for each $J$, $Z_J$ is a smooth subvariety of $\bar X$. The condition that $\bar D$ intersects $Z$ in good position means that $\bar D$ must intersect each stratum $Z_J$  transversely, i.e., that $Z_J\cap \bar D$ is a smooth subvariety of $\bar D$ of dimension $n-1-|J|$.

\end{itemize}
Then the positive density (and hence non-emptiness) of $\mathcal T$ immediately follows from are \cite[Theorem 1.3]{poonen2004bertini}: the conditions on $f$ are that $V(f)\cap \bar X$ intersect a finite set of smooth subvarieties transversely, avoid a given finite set of points, pass through another given finite set of points, and have sufficiently high degree by Proposition \ref{proposition:divisor_absorbs_sections}.
\end{proof}
Note that Lemma \ref{lem:good_curves_exist} also holds with $\Fq$ replaced by any infinite field $k$ by the usual Bertini theorems.
Finally, the following is surely well-known but we could not find a reference for exactly the statement we need. (The essential content is contained in \cite[Section 3.3]{chai2014complex}.) We will use this lemma to make a particular choice of $A_C\rightarrow C$ in the isogeny class from Drinfeld's Theorem \ref{Theorem:GL2} (though this choice will not be unique).
\begin{lem}\label{lem:isogeny_class}Let $X$ be a scheme and let $A\rightarrow X$ be an abelian scheme. Let $r$ be a prime and let $G$ be an $r$-divisible group on $X$. Suppose there exists an isogeny $\psi\colon A[r^{\infty}]\rightarrow G$ of $r$-divisible groups on $X$ (as in \cite[3.3.5]{chai2014complex}). Then there exists an $r$-primary isogeny $\varphi\colon A\rightarrow B$ of abelian schemes over $X$ and an isomorphism $\varepsilon\colon B[r^{\infty}]\rightarrow G$ such that the following diagram commutes.
$$
\xymatrix{
A[r^{\infty}]\ar[rd]_{\varphi[r^{\infty}]}\ar[rr]^{\psi} & & G\\
&              B[r^{\infty}]\ar[ru]_{\varepsilon}       &
}$$
\end{lem}
\begin{proof}
Set $N=\text{ker}(\psi)$. Then $N$ is a (commutative) finite flat group scheme over $X$ of $r$-primary order. We have a short exact sequence in the category of fppf sheaves:
$$0\rightarrow N\rightarrow A[p^{\infty}]\xrightarrow{\psi}  G\rightarrow 0.$$
Consider the quotient $A/N$ in the category of fppf sheaves. It follows from e.g. \cite[1.4.1.3, 1.4.1.4]{chai2014complex} that there exists an abelian scheme $B\rightarrow X$ that represents the sheaf $A/N$. We then have the following commutative diagram of fppf sheaves.
$$\xymatrix{
0\ar[r]	&N\ar[r]\ar[d]_{\cong}	&A[r^{\infty}]\ar[r]^{\psi}\ar[d]	&G\ar[r]\ar@{.>}[d]	&0\\
0\ar[r]	&N\ar[r]	&A\ar[r]					&B\ar[r]	&0
}$$
where the right vertical arrow exists because $G=\text{coker}(N\rightarrow A[r^{\infty}])$. We claim that the induced map $G\rightarrow B$ yields an isomorphism $G\rightarrow B[r^{\infty}]$. By the snake lemma, $G\rightarrow B$ is injective. However, an injective isogeny of $r$-divisible groups is an isomorphism.

\end{proof}

\section{Proofs of Theorem \ref{main_thm} and Corollaries \ref{cor:higher_dimension_drinfeld}, \ref{cor:higher_dimensional_padic_drinfeld}}

\begin{proof}[Proof of Theorem \ref{main_thm}]We proceed in several steps.

\textbf{Step 1: \emph{organizing the summands of $\mathcal E$.}} As $\calE_i$ is irreducible, has determinant $\Qpbar(-1)$, and has rank 2, the slopes of $(\calE_i)_x$ are in the interval $[0,1]$ for every closed point $x$ of $X$, see \cite[Section 1.2, p. 136-137]{drinfeld2016slopes}, where it is deduced from Corollary 1.1.7 of \emph{loc. cit.}.

Write the isotypic decomposition of $\calE$ in $\fisocd{X}_{\Qpbar}$:
$$\displaystyle \calE\cong \bigoplus_{i=1}^a(\calE_i)^{m_i}.$$

The field generated by the coefficients of $P_x(\calE,t)$ as $x$ ranges through closed points of $X$ is $\QQ$. Therefore, by \cite[E.10]{drinfeld2018pro} and either \cite[Theorem 4.2]{abe2016lefschetz} or \cite[Corollary 3.5.3]{kedlayacompanions}, we can pick an $l$ and a field isomorphism $\sigma\colon \Qpbar\rightarrow \Qlbar$ such that the semi-simple $\sigma$ companion $L$ to $\calE$ exists and in fact may be defined over $\Ql$, i.e., corresponds to a representation:
$$\pi_1(X)\rightarrow \text{GL}_N(\Ql).$$
(We emphasize that $L$ is independent of the choice of $\sigma$ by Proposition \ref{prop:shape_ell_companion}.) By compactness of $\pi_1(X)$, we may conjugate the representation into $\text{GL}_N(\ZZ_l)$. We refer to the attached lisse $\ZZ_l$-sheaf as $\tilde{L}$.  Similarly, for each $i$ we denote by $L_i$ the $\sigma$-companion to $\calE_i$ (the $L_i$ indeed do depend on the choice of $\sigma$). The companion relation commutes with direct sum; hence we have:
$$\displaystyle L\cong \bigoplus_{i=1}^a L_i^{m_i}.$$

\noindent (See also the proof of Proposition \ref{prop:shape_ell_companion}.) Let $E_i\subset \Qpbar$ denote the (number) field generated by the coefficients of $P_x(\calE_i,t)$ as $x$ ranges through the closed points of $X$. Note that for each $\calE_i$, all $p$-adic companions exist and are summands of $\calE$ by Lemma \ref{lem:all_companions}. For each $\calE_i$, set $\calF_i$ to be the sum of all distinct $p$-adic companions of $\calE_i$. Note that there are $[E_i\colon \QQ]$ distinct $p$-adic companions of $\calE_i$, parametrized by the embeddings $E_i\hookrightarrow \Qpbar$. By reordering the indices, we write the decomposition of $\calE$ as follows:

\begin{equation}\label{eqn:complete_companions_decomposition}
\displaystyle \calE\cong \bigoplus_{i=1}^b \calF_i^{m_i}
\end{equation}
for some integer $1\leq b\leq a$. (Under this reordering, the collection of $(\calE_i)_{i=1}^b$ are all mutually not companions and for each $b+1\leq j \leq a$, there exists a unique $1\leq i\leq b$ such that $\calE_j$ is a companion of $\calE_i$.) Set

\begin{equation}\label{equation:g}g=\sum_{j=1}^b m_i[E_i\colon \QQ].\end{equation}

\textbf{Step 2: \emph{the proof in a simplified situation.}} We first assume that $X$ admits a simple normal crossings compactification $\bar{X}$ such that $\calE$ extends to a logarithmic $F$-isocrystal $\bar \calE$ with nilpotent residues on $\bar{X}$ and moreover that $\tilde{L}$ has trivial residual representation. Write $Z:=\bar{X}\backslash X$ for the boundary. (Note that under the above assumption on $\calE$, the $l$-companion $L$ is tamely ramified.)

By Lemma \ref{lem:katz}, there exists a Zariski open $\bar{U}\subset \bar{X}$ with complementary codimension at least 2, and a logarithmic Dieudonn\'e crystal $(M_{\bar U},F,V)$ on $\bar{U}$ (with the logarithmic structure coming from $Z\cap \bar{U}$) such that the associated logarithmic $F$-isocrystal is isomorphic to $\bar{\calE}|_{\bar U}$. (In other words, $M_{\bar U}$ is an $F$ and $p\circ F^{-1}$ stable lattice in $\bar{\calE}|_{\bar U}$.) Let
$$(N_{\bar U},F,V):=(M_{\bar U},F,V)^4\oplus ((M_{\bar U},F,V)^t)^4,$$
where the $t$ denotes the \emph{dual logarithmic Dieudonn\'e crystal}. We also consider this logarithmic Dieudonn\'e crystal as we will need to use Zarhin's trick. We set $U:=\bar{U}\backslash (\bar{U}\cap Z)$.

After Remark \ref{rmk:hodge}, it follows that we may define Hodge line bundles $\omega_M$ and $\omega_N$ on $\bar{U}$ attached to the two logarithmic Dieudonn\'e crystals. As $\bar{U}\subset \bar{X}$ has complementary codimension at least 2 and $\bar{X}$ is smooth, it follows that $\omega_M$ and $\omega_N$ extend canonically to line bundles on all of $\bar{X}$.

The Hodge line bundle $\alpha$ on the fine moduli scheme $\mathscr{A}_{8g,1,l}\otimes \Fq$ is ample by \cite[Ch. IX, Th\'eor\`eme 3.1, p. 210]{moret1985pinceaux} or \cite[Ch. V, Theorem 2.5(i)]{faltings2013degeneration}.  Let $g$ be as in Equation \ref{equation:g} and choose an $r$ so that the $\alpha^r$ is very ample on $\mathscr{A}_{8g,1,l}$. As $8g>1$, it follows from the Koecher principle  that $H^0(\mathscr{A}_{8g,1,l}\otimes \Fq,\alpha^r)$ is a finite dimensional $\Fq$-vector space for all $r\in \mathbb{Z}$ \cite[Ch. V, Theorem 1.5 (ii)]{faltings2013degeneration}. Fix a basis $s_0,\dots,s_m$ of the vector space:

\begin{equation}\label{eqn:fix_basis}s_0,\dots,s_m\in H^0(\mathscr{A}_{8g,1,l}\otimes \Fq,\alpha^r)\end{equation}
once and for all. There is an induced embedding $\mathscr{A}_{8g,1,l}\subset \mathbb{P}^m$. As is customary, denote by $\mathscr{A}^*_{8g,1,l}$ the Zariski closure of $\mathscr{A}_{8g,1,l}$ in $\mathbb{P}^m$; we call this the \emph{minimal compactification}. Abusing notation, we also denote by $\alpha$ the Hodge line bundle on $\mathscr{A}^*_{8g,1,l}$. The Koecher principle implies that $H^0(\mathscr{A}_{8g,1,l}\otimes \Fq,\alpha^r)=H^0(\mathscr{A}^*_{8g,1,l}\otimes \Fq,\alpha^r)$: this follows from  \cite[Ch. V, Theorem 1.5 (ii), Theorem 2.5 (iii)]{faltings2013degeneration}.

It follows from \cite[Proposition 3.4]{deligne2012finitude}  there exists a \emph{finite} number of closed points $(x_j)^{s}_{j=1}$ of $U$ such that for each $\calE_i$, the field generated by the coefficients of $P_{x_j}(\calE_i,t)\in \Qpbar[t]$ as $j=1\dots s$ is $E_i\subset \Qpbar$. We call this fact $\blacklozenge$.

If $\bar{C}\subset \bar{U}$ is a good curve for the quintuple $(\bar X, X, \bar U, \omega_N^r, (x_j)^s_{j=1})$ as in Definition \ref{def:good_curve}, set $C:=\bar{C}\cap X$. Then the following three properties hold.
\begin{itemize}
\item Each $\calE_i|_C$ is irreducible by \cite[Theorem 2.6]{abe2016lefschetz}.
\item The field generated by Frobenius traces of $\calE_i|_C$ is $E_i$ by $\blacklozenge$.
\item Each $\calE_i|_C$ has infinite monodromy around $\infty$. Indeed, from the positivity of $\bar{C}$, and the good position assumption, it follows that $\bar C$ intersects each irreducible component $Z_m$ of $Z$ in a non-empty and transverse way and moreover $\bar C$ does not intersect the codimension 2 strata $Z_{m}\cap Z_n$. By assumption, for each $\calE_i$, there exists a component $Z_m$ around which the monodromy around $Z_m$ of $\calE_i$ (equivalently, of $L_i$) is infinite. On the other hand, there is a surjective morphism of tame fundamental groups $$\pi_1^{\text{tame}}(C)\twoheadrightarrow \pi_1^{\text{tame}}(X)$$ by \cite[Theorem 1.1(a)]{esnaultkindler}. Moreover, for each $m$, we may restrict the above surjection to a surjective map of tame inertia groups
$$I_{Z_m\cap \bar C}^{\text{tame}}(C)\twoheadrightarrow I_{Z_m}^{\text{tame}}(X)$$
around $Z_m\cap \bar C$ and $Z_m$ respectively. By the  assumption that $L_i$ had infinite monodromy around $Z_m$ and the fact that wild inertia is a pro-$p$ group, it follows that the image of $I_{Z_m}^{\text{tame}}(X)$ in the $l$-adic representation corresponding to $L_i$ is infinite. Therefore, the image of $I_{Z_m\cap \bar C}^{\text{tame}}(C)$ in the $l$-adic representation corresponding to $L_i|_C$ is also infinite, or equivalently, $\calE_i|_C$ has infinite monodromy around $Z_m\cap \bar C$, as desired.
\end{itemize}

 Let $\bar{C}\subset \bar{U}$ be a good curve for the quintuple $(\bar X, X, \bar U, \omega_N^r, (x_j)^s_{j=1})$. Recall the decomposition from Equation \ref{eqn:complete_companions_decomposition}: $\displaystyle \calE\cong \oplus_{i=1}^b \calF_i^{m_i}$, where each $\calF_i$ is the sum of the distinct companions of $\calE_i$ under the reordering specified in Step 1. (Note that $\calF_i$ has Frobenius traces in $\QQ$.) By Theorem \ref{Theorem:GL2} and Remark \ref{rem:mult_1}, for each $i\in \{1,\dots,b\}$, there exists an abelian scheme $A_i\rightarrow C$ of dimension $g_i=[E_i:\QQ]$ such that $\calF_i|_C$ is compatible with $A_i$. By taking the iterated fiber product over $C$, it therefore follows from Equation \ref{equation:g} that there exists an abelian scheme $\pi_C\colon A_C\rightarrow C$ of relative dimension $g$ such that
$$R^1(\pi_C)_*\Qlbar\cong L|_C.$$
As $l$ is prime to $p$, it follows from the Galois correspondence for $\pi_1(X)$ that the category of (necessarily \'etale) $l$-divisible groups on $X$ is equivalent to the category of lisse $\Zl$ sheaves on $X$ (see e.g. \cite[Pages 147-148]{chai2014complex}, where they explain that the functor is explicitly given as the Tate $l$-group). Write $\Phi$ for an inverse functor. We have assumed that the $\ZZ_l$-lattice $\tilde{L}$ has trivial residual representation, i.e., the following map is trivial
$$\pi_1(X)\rightarrow \text{GL}_{2g}(\ZZ/l\ZZ).$$
Then it follows from the Galois correspondence that $\Phi(\tilde{L})[l]$ is isomorphic to the split \'etale group scheme $(\ZZ/l\ZZ)^{2g}$ (see e.g. the explicit formula on p. 148 of \emph{loc. cit.}). On the other hand, $\Phi(\tilde{L})$ is isogenous to $A_C[l^{\infty}]$ because $\tilde{L}$ is isogenous to $T_l(A_C)$. It follows from Lemma \ref{lem:isogeny_class} that there exists an $l$-primary isogeny $A_C\rightarrow A'_C$ over $C$ such that $A'_C[l]\rightarrow C$ is isomorphic to the split \'etale group scheme $(\ZZ/l\ZZ)^{2g}$.  The abelian scheme $A'_C\rightarrow C$ therefore has a full collection of $l$-torsion sections, i.e., it has trivial $l$-torsion.  Replacing $A_C$ by $A_C'$, we may assume that $A_C$ has trivial $l$-torsion. 

Similarly, we claim that $\mathbb{D}(A_C[p^{\infty}])\otimes \Qpbar\cong \calE|_C$. Indeed, $\mathbb{D}(A_C[p^{\infty}])\otimes \Qpbar$ is a semisimple object of $\fisocd{C}_{\Qpbar}$ by \cite{pal2015monodromy} and is compatible with $L|_C$ by \cite{katzmessing}. Therefore $\mathbb{D}(A_C[p^{\infty}])$ is isogenous to $(M,F,V)_C$ as Dieudonn\'e crystals on $C$. We claim that we may replace $A_C$ by an ($p$-primarily) isogenous abelian scheme in order to ensure that:

$$\mathbb{D}(A_C[p^{\infty}])\cong (M,F,V)_C$$
as Dieudonn\'e crystals on $C$. To see this, use \cite{de1995crystalline} to construct a $p$-divisible group $G_C$ on $C$ where $\mathbb{D}(G_C)\cong (M_C,F,V)$. It follows that $A_C[p^{\infty}]$ and $G_C$ are isogenous. Applying Lemma \ref{lem:isogeny_class}, we see that there is a $p$-primary isogeny $A_C\rightarrow A'_C$ such that $A'_C[p^{\infty}]\cong G_C$. As the group of $l$-torsion points of an abelian scheme is a finite flat $l$-primary group scheme, it follows that $A'_C$ also has trivial $l$-torsion. Replace $A_C$ by $A'_C$. We emphasize that this choice of $A_C$ is not canonical!

By construction, the $l$-torsion of $A_C\rightarrow C$ is trivial; it follows that $A_C\rightarrow C$ has semistable reduction along $\bar C\cap Z$. (Use that the monodromy representation $\pi_1(C)\rightarrow GL_{2g}(\mathbb Z_l)$ has image in $\Gamma(l):=\{1+M|\ M\in lM_{n\times n}(\mathbb Z_l)\}\subset GL_{2g}(\mathbb Z_l)$, and the fact that if $l>2$, the group $(1+l\overline{\mathbb Z}_l)^{\times}$ is torsion free. Therefore, if $\gamma\in \pi_1(C)$ has quasi-unipotent image in the representation, it then in fact has unipotent image. The claim then follows from Grothendieck's semistable reduction theorem for abelian varieties.)

Let $A_{\bar C}\rightarrow \bar C$ be the N\'eron model and let $A^o_{\bar C}\rightarrow \bar C$ denote the associated semi-abelian scheme, i.e., the open subset of $A_{\bar C}\rightarrow C$ obtained by removing the non-identity components along $\bar C\setminus C$. It follows from the third part of Proposition \ref{prop:log_Hodge_compatible} that the logarithmic Dieudonn\'e crystal of $A_{\bar C}\rightarrow \bar{C}$ constructed in Remark \ref{rem:Hodge_compatible} is isomorphic to $(M,F,V)_{\bar C}$. Then by the second part of Proposition \ref{prop:log_Hodge_compatible}, the Hodge bundle of the $A^o_{\bar C}\rightarrow \bar{C}$ is isomorphic to $\omega_M|_{\bar C}$. 

Set $B_C:=(A_C\times_C A^t_C)^4$. By Zarhin's trick \cite[Chapitre IX, Lemme 1.1, p. 205]{moret1985pinceaux}, $B_C$ admits a principal polarization. By construction, we have that
\begin{itemize}
\item $B_C$ has trivial $l$-torsion, and
\item $\mathbb{D}(B_C[p^{\infty}])\cong (N_C,F,V)$
\end{itemize}
By the uniqueness part of Proposition \ref{prop:log_Hodge_compatible} it follows that there is an isomorphism of logarithmic Dieudonn\'e crystals: $$\mathbb{D}^{\log}(B_{\bar C})\cong (N,F,V)_{\bar C}.$$  The Hodge line bundle of $B^o_{\bar C}\rightarrow \bar C$ is hence isomorphic to $\omega_N|_{\bar C}$ again by Proposition \ref{prop:log_Hodge_compatible}. However, we emphasize again that the choice $B_C\rightarrow C$ is not canonical!

We have an induced morphism to a fine moduli scheme $C\rightarrow \mathscr{A}_{8g,1,l}$. This extends to a morphism
\begin{equation}\label{eqn:map_to_compactification}
\xymatrix{ C\ar[r]\ar[d] &\mathscr{A}_{8g,1,l}\ar[d]\\
\bar{C}\ar[r]			& \mathscr{A}_{8g,1,l}^*
}
\end{equation}to the minimal compactification because $\mathscr{A}_{8g,1,l}^*/\mathbb F_q$ is proper and $\bar C/\Fq$ is a smooth curve. We now claim the pullback of $\alpha$, the Hodge line bundle on $\mathscr{A}^*_{8g,1,l}$, is isomorphic to $\omega_N|_{\bar C}$. Here is the reason: choose a toroidal compactification $\bar{\mathscr A}_{8g,1,l}$. We then have a commutative diagram:
\begin{equation}
\xymatrix{
\bar C\ar[dr]\ar[r]^h & \bar{\mathscr{A}}_{8g,1,l}\ar[d]^{\varphi}\\
 & \mathscr{A}^*_{8g,1,l},
}
\end{equation}
again, because $\bar{\mathscr{A}}_{8g,1,l}/\mathbb F_q$ is proper and $\bar C/\Fq$ is a smooth curve. By \cite[Ch. V, Theorem 2.5]{faltings2013degeneration}, there is a semi-abelian scheme $G\rightarrow \bar{\mathscr{A}}_{8g,1,l}$ such that $\varphi^*\alpha$ is isomorphic to the Hodge line bundle of $G\rightarrow \bar{\mathscr{A}}_{8g,1,l}$. Now, \cite[Ch. I, Proposition 2.7]{faltings2013degeneration} implies that $h^*G$ is isomorphic to $A^o_{\bar C}\rightarrow \bar C$, i.e., the semi-abelian scheme given by the open subset of $A_{\bar C}\rightarrow \bar C$ obtained by removing the non-identity components along $\bar{C}\backslash C$. In particular, it follows from part (2) of Proposition \ref{prop:log_Hodge_compatible} that the Hodge line bundle of $h^*G$ is compatible with the Hodge line bundle constructed in Remark \ref{rmk:hodge}.

In Equation \ref{eqn:fix_basis}, we have already fixed a basis of sections $$s_0,\dots,s_m \in H^0(\mathscr{A}_{8g,1,l}\otimes \Fq,\alpha^r)=H^0(\mathscr{A}^*_{8g,1,l}\otimes \Fq,\alpha^r);$$

after pulling back the sections to $\bar{C}$ via Equation \ref{eqn:map_to_compactification}, we obtain an $m+1$-tuple of sections $t_0,\dots,t_m$ in $H^0(\bar{C},\omega^r_N|_{\bar C})$ that define the morphism $\bar{C}\rightarrow \mathscr{A}^*_{8g,1,l}\subset \mathbb{P}^m$.

In conclusion, for every good curve $\bar{C}\subset \bar{U}$ for the quintuple $(\bar X, X, \bar U, \omega_N^r,(x_j)^s_{j=1})$, we have constructed an $m+1$-tuple of globally generating sections $t_0,\dots,t_m\in H^0(\bar{C},\omega^r_N|_{\bar C})$ such that
\begin{itemize}
\item the induced map lands in $\mathscr{A}^*_{8g,1,l}\subset \mathbb{P}^m$;
\item the image of $C$ under the induced map lands in $\mathscr{A}_{8g,1,l}\subset \mathscr{A}_{8g,1,l}^*$;
\item and such that the induced abelian variety on $B_C\rightarrow C$ is isomorphic to $(A_C\times_C A_C^t)^4$ where $A_C\rightarrow C$ is an abelian scheme with $\mathbb{D}(A_C[p^{\infty}])\cong (M,F,V)|_C$ as Dieudonn\'e crystals on $C$. (Therefore we also have that $\mathbb{D}(B_C[p^{\infty}])\cong (N,F,V)|_C$.)
 \end{itemize}
 In particular, setting $M=\mathscr{A}_{8g,1,l}^*\subset \mathbb P^m$, condition (3) of Lemma \ref{lem:map_on_curves} holds for $\bar{C}\subset \bar{X}$ (corresponding to the symbols $C\subset Y$ in Lemma \ref{lem:map_on_curves}). Note that for two such good curves $C$ and $C'$, there is no reason that the induced maps to $\mathscr{A}_{8g,1,l}$ match up on the intersection $C\cap C'$ because our choices of abelian schemes were not canonical.

For each $n>0$, let $P_n$ denote the union of the set of closed points of $U$ whose residue field is contained in $\FF_{q^{n!}}$. Note that for any infinite subset $S\subset \mathbb N$, the set $\bigcup_{n\in S}P_n$ is Zariski dense in $X$; indeed, any given closed point $x$ of $U$ is an element of $P_n$ for all $n\gg 0$. By Lemma \ref{lem:good_curves_exist}, it follows that for each $n>0$, there exists a good curve $\bar C_n\subset \bar U$ for the quintuple $(\bar X, X, \bar U, \omega_N^r, P_n)$. 

For each $n\in \mathbb{N}$, by the above remarks we obtain an $m+1$-tuple of globally generating sections $$t_{n,0},\dots, t_{n,m}\in H^0(\bar{C}_n,\omega^r_N|_{\bar{C}_n})$$
such that the induced map factors $f_n\colon \bar{C}_n\rightarrow \mathscr{A}^*_{8g,1,l}\subset \mathbb{P}^m$. Moreover, any infinite subcollection of the $\bar{C}_n$ is Zariski dense because they are space-filling: if $x$ is a closed point of $U$ with residue field $\mathbb F_{q^e}$, then $x$ is contained in $C_n$ for all $n\geq e$. By Lemma \ref{lem:map_on_curves}, it follows that there exists an infinite set $S\subset \mathbb{N}$ and sections $\tilde{t}_0,\dots,\tilde{t}_m\in H^0(\bar{X},\omega^r_N)$ such that the induced rational map $\tilde{f}\colon \bar{X}\dashrightarrow \mathbb{P}^m$ lands in $\mathscr{A}^*_{8g,1,l}$ and moreover, for each $n\in S$, we have an equality of morphisms $\tilde{f}|_{\bar{C}_n}=f_n$.

By shrinking $U$, we therefore obtain a map $\tilde{f}\colon U\rightarrow \mathscr{A}_{8g,1,l}$ and hence an abelian scheme $B_U\rightarrow U$ such that $B_U[l]$ is a trivial \'etale cover of $U$. The maps $f_n\colon \bar{C}_n\rightarrow \mathscr{A}_{8g,1,l}^*$ were all constructed such that the induced abelian scheme $B_{C_n}\rightarrow C_n$ is compatible with $$(N_{C_n},F,V)\otimes \Qp\cong (\calE\oplus \calE^*(-1))^4|_{C_n}.$$ On the other hand, if $u$ is a closed point of $U$, then $u$ lies on $C_n$ for all $n\gg 0$. We claim that it follows that $B_U\rightarrow U$ is compatible with $(L\oplus L^*(-1))^4|_U$. Indeed, it suffices to show that for every closed point $u$ of $U$, $B_u\rightarrow u$ and $(L\oplus L^*(-1))^4|_u$ are compatible, i.e., that the characteristic polynomials of Frobenius match up. Pick $n\in S$ with $C_n$ containing $u$. As the map $\tilde{f}\colon U\rightarrow \mathscr{A}_{8g,1,l}$ extends the map $f_n\colon C_n\rightarrow \mathscr{A}_{8g,1,l}$ by the definition of $S$ in Lemma \ref{lem:map_on_curves}, the induced abelian scheme $B_U\rightarrow U$ extends the abelian scheme $B_{C_n}\rightarrow C_n$ constructed above, which is compatible with $(L\oplus L^*(-1))^4|_{C_n}$ by construction. Therefore $B_u\rightarrow u$ is compatible with $(L\oplus L^*(-1))^4|_{u}$, as desired.

For each $n\in S$ we have that $\tilde{f}|_{\bar{C}_n}=f_n$. By construction there exists an abelian scheme $A_{C_n}\rightarrow C_n$ of dimension $g$ with $$B_U|_{C_n}\cong (A_{C_n}\times_{C_n}A^t_{C_n})^4.$$ Consider the map of representations induced by the first $\ZZ_l$-cohomology of the abelian schemes $B_U\rightarrow U$ and $B_U|_{C_n}\rightarrow C_n$:
\begin{equation}\label{equation:reps_same_image}
\xymatrix{\pi_1(C_n)\ar[dr]\ar[rr] && \pi_1(U)\ar[dl]\\
 & \text{GL}_{16g}(\ZZ_l).&
}
\end{equation}
Then \cite[Lemma 6(b)]{katz2001spacefilling} implies that for $n\gg 0$, the two representations have the same image (which lands in $\text{GL}_{2g}(\ZZ_l)^8$). By the fundamental work of Tate-Zarhin on Tate's isogeny theorem for abelian varieties over finitely generated fields of positive characteristic \cite[Ch. XII, Th\'eor\`eme 2.5(i), p. 244]{moret1985pinceaux}, it follows that for all $n\gg 0$ the natural injective map $\text{End}_U(B_U)\hookrightarrow \text{End}_{C_n}({B_U|_{C_n}})$ is an isomorphism when tensored with $\ZZ_l$ and hence also when tensored with $\QQ_l$. Therefore, for all $n\gg 0$, the map $$\text{End}_U(B_U)\otimes \QQ\rightarrow \text{End}_{C_n}(B_U|_{C_n})\otimes \QQ$$ is an isomorphism as both sides are finite dimensional semi-simple $\QQ$-algebras of the same rank.

We know that $\text{End}_{C_n}(B_{U}|_{C_n})$ has a \emph{nontrivial idempotent} $e_{C_n}$ that projects onto a copy of $A_{C_n}$. After replacing $e_{C_n}$ by a high integer multiple, we may lift $e_{C_n}$ to $e_U\in \text{End}_U(B_U)$. Set the image of $e_U$ to be the abelian scheme $\pi_U\colon A_U\rightarrow U$. Finally, we claim that $A_U$ is compatible with $L$ (equivalently: $\calE$). Indeed, the image $A_U\rightarrow U$ is an abelian scheme of dimension $g$ that extends $A_{C_n}\rightarrow C_n$. On the other hand, in Equation \ref{equation:reps_same_image} the two images in $\text{GL}_{16g}(\Zl)$ are the same (as we have assumed $n\gg 0$) and hence have corresponding decompositions in irreducible $\Ql$ representations.

\textbf{Step 3: \emph{the proof in the general case via reduction to Step 2}.} There exists a projective divisorial compactification $\bar{X}$ of $X$. (This means that $\bar{X}$ is normal and the boundary is an effective Cartier divisor.) By Kedlaya's semistable reduction theorem (see \cite[Theorem 7.6]{kedlaya2016notes} for a meta-reference), there is a generically \'etale alteration $\varphi\colon X'\rightarrow X$ together with a simple normal crossings compactification $\bar{X'}$ such that the overconvergent pullback $\calE'$ extends to a logarithmic $F$-isocrystal with nilpotent residues.  After replacing $X'$ with a further finite \'etale cover, we may guarantee that the residual representation of $L'$ is trivial.

%

 We have proven the theorem for $\calE'$ on $X'$: there exists an open subset $W'\subset X'$ and an abelian scheme $A_{W'}\rightarrow W'$ of relative dimension $g$ with $\mathbb{D}(A_{W'}[p^{\infty}])\cong \calE'|_{W'}$. After shrinking $W'$ and $W$, we may assume that $\varphi|_{W'}\colon W'\rightarrow W$ is finite \'etale, of degree $d$.

Set $B_W:=\mathfrak{Res}^{W'}_{W}(A_{W'})$ to be the Weil restriction of scalars. This is an abelian scheme over $W$ of dimension $dg$. We claim that $B_W$ is compatible with $L^{d}$. One way to see this is the following. Consider the short exact sequence of abelian sheaves in the \'etale topology:

$$0\rightarrow A_{W'}[l^n]\rightarrow A_{W'}\xrightarrow{l^n}A_{W'}\rightarrow 0.$$
As $W'\rightarrow W$ is finite \'etale, Weil restriction is exact on the level of abelian \'etale sheaves. Therefore $\mathfrak{Res}^{W'}_W(A_{W'}[l^n])\cong B_W[l^n]$. As $A_{W'}[l^n]$ is a finite \'etale group over $W$, one deduces that the representation associated to $B_W$ is isomorphic to

$$\text{Ind}_{\pi_1(W)}^{\pi_1(W')}L'.$$
However, $L'$, as a representation, is the restriction of $L$ along the inclusion $\pi_1(W')\hookrightarrow \pi_1(W)$. Then the desired compatibility follows from the following fact: if $H\subset G$ is the inclusion of a subgroup of finite index $d$, and if $V$ is a finite dimensional representation of $G$, then $$\text{Ind}_H^G\text{Res}^G_H(V)\cong V^{\oplus d}.$$ 
Recall that we wrote an isotypic decomposition:
$$\displaystyle L\cong \bigoplus_{i=1}^a (L_i)^{m_i}$$
where each $L_i$ is irreducible on $X$ (and hence on $W$ by \cite[Lemma 1.1.2]{kedlayacompanions}). Let $E_i\subset \Qlbar$ denote the field generated by the traces of Frobenius on $L_i$ as $x$ ranges through the closed points of $W$. 
We claim that we may find a smooth curve $C\subset W$ with the following properties:

\begin{enumerate}
\item each $L_i|_C$ is irreducible;
\item the field generated by Frobenius traces of $L_i|_C$ is $E_i\subset \Qlbar$;
\item each $L_i|_{C}$ has infinite monodromy around $\infty$; and
\item the induced monodromy representations coming from $B_W\rightarrow W$ and $B_W|_C\rightarrow C$ \[
\xymatrix{\pi_1(C)\ar[dr]\ar[rr] && \pi_1(W)\ar[dl]\\
 & \text{GL}_{2gd}(\ZZ_l)&
}
\]
have the same image.
\end{enumerate}

We have a projective normal compactification $\bar{X}$ of $X$, which is smooth away from a closed subset of codimension at least 2. Let $F=\bar{X}\backslash X$ and let $F'\subset F$ be the singular locus of $\bar{X}$. For each $L_i$, there is an irreducible component $F_j$ of $F$ that witnesses the fact that $L_i$ has infinite monodromy at $\infty$: having infinite monodromy at $\infty$ means that a certain inertia group has infinite image in the representation.

Pick a closed point $y_j\in F_j\backslash (F_j\cap F')$ for each $j$. Then, by using \cite[C.2]{drinfeld2012conjecture}, we may construct an infinite set of curves $(C_n)_{n\in \NN}$ where each $C_n\subset W$ is a smooth, geometrically connected curve that contains all closed points of $W$ whose residue fields are contained in $\FF_{q^{n!}}$ and that pass through the $y_j$ transversally (i.e., with a tangent direction that is not contained in $F_j$). (We remark that this is a consequence of Poonen's Bertini theorem \cite[Theorem 1.3]{poonen2004bertini}.)

Each $L_i|_{C_n}$ has infinite monodromy around $\infty$. By \cite[Lemma 6(b)]{katz2001spacefilling}, it follows that for all $n\gg 0$, $C_n$ satisfies (4). For $n\gg 0$, \cite[Lemma 6(b)]{katz2001spacefilling} and \cite{deligne2012finitude} guarantees that setting $C:=C'_n$ satisfies the above four conditions.

Again, by using Drinfeld's Theorem \ref{Theorem:GL2}, Remark \ref{rem:mult_1}, and Equation \ref{eqn:complete_companions_decomposition} as in Step 2, there exists an abelian scheme $A_{C}\rightarrow C$ that is compatible with $L|_{C}$. On the one hand, using the Tate isogeny theorem \cite[Ch. XII, Th\'eor\`eme 2.5]{moret1985pinceaux} it follows that $A^d_C$ is isogenous to $B_W|_C$. On the other hand, another application the Tate isogeny theorem together with property (4) of $C$ implies that the natural map:
$$\text{End}_{W}(B_W)\rightarrow \text{End}_{C}(B_W|_C)$$
is an isomorphism after tensoring with $\QQ$. As $B_W|_C$ is isogenous to $A_C^d$, it follows that $\text{End}_{C}(B_W|_C)\otimes \QQ$ has an element $e_{C}$ projecting onto a factor of $A_C$. After replacing $e_C$ with a high integer multiple, we may lift to $e_W\in \text{End}_W(B_W)$. Set the image of $e_W$ to be the abelian scheme $A_W\rightarrow W$; this is compatible with $L|_W$, as desired.
\end{proof}
\begin{proof}[Proof of Corollary \ref{cor:higher_dimension_drinfeld}]
Suppose there exists $\pi_U\colon A_U\rightarrow U$ such that $R^1(\pi_U)_*\Qlbar$ has $L_1$ as a summand. By the assumption that $X$ is smooth and geometrically connected, it is irreducible; hence $U\subset X$ is dense. A theorem of Zarhin implies that $R^1(\pi_U)_*\Ql$ is semi-simple \cite[Chapitre XII, Theorem 2.5, p. 244-245]{moret1985pinceaux}. The field generated by the characteristic polynomials of $R^1(\pi_U)_*\Ql$ is clearly $\QQ$; indeed, this follows Weil's theorem that the characteristic polynomial of Frobenius acting on the Tate module of an abelian variety over a finite field has coefficients in $\mathbb Z$ \cite[IX,X]{weilAV}.

We claim that $\mathbb{D}(A_U[p^{\infty}])\otimes \Qp$ is a semi-simple object of $\fisocd{U}$. This is essentially contained in \cite[Remark 4.11]{pal2015monodromy}, but some comments are in order.

 While the statement of Remark 4.11 of \emph{loc. cit.} assumes that $U$ is a smooth curve, this assumption is unnecessary. Indeed, the only point where this assumption is used is in the citation to \cite[4.3-4.8]{katotrihan}, to argue that the associated $F$-isocrystal is overconvergent. However, \cite[Th\'eor\`eme 7]{etesse} essentially states and proves exactly this: if $S/k$ is a smooth separated scheme over a field $k$ of characteristic $p$ and $A\rightarrow S$ is an abelian scheme, then $R^1f_{\text{rig}}(\mathcal O_{X/K})$ is an overconvergent $F$-isocrystal on $S$. When $k$ is perfect, this $F$-isocrystal is isomorphic to $R^1f_{\text{crys}}(\mathcal O_{X/W})\otimes \mathbb Q$ because $A\rightarrow S$ is smooth and proper (see e.g. \cite[Proposition 1.9]{berthelot1997finitude}). On the other hand,  $R^1f_{\text{crys}}(\mathcal O_{X/W})\cong \mathbb D(A[p^{\infty}])$. In particular, to obtain the semisimplicity, one simply combines Corollary 4.9 and Proposition 3.5 of \cite{pal2015monodromy} with the fact that $\mathbb D(A_U[p^{\infty}])\in \fisocd{U}$, exactly as explained in Remark 4.11 of \emph{loc. cit.}
 
As $\mathbb{D}(A_U[p^{\infty}])\otimes \Qp$ is isomorphic to the rational crystalline cohomology of $A_U\rightarrow U$, it follows from \cite{katzmessing} that that $\mathbb{D}(A_U[p^{\infty}])\otimes \Qp$ and $R^1(\pi_U)_*\Ql$ are companions.  It follows from Lemma \ref{lem:all_companions} that all crystalline companions of $L_1|_U$ exist and moreover are summands of $\mathbb{D}(A_U[p^{\infty}])\otimes \Qpbar$. Then by \cite[Corollary 3.3.3]{kedlayacompanions}, all crystalline companions to $L_1$ exist.

Conversely, suppose all crystalline companions $(\calE_i)^b_{i=1}$ to $L_1$ exist. We first of all claim that each $\calE_i$ has infinite monodromy at $\infty$. Indeed, suppose for contradiction that there existed an alteration $f\colon X'\rightarrow X$ and a compactification $\overline{X'}$ such that $f^*\calE_i$ extends to an object $\calF'$ of $\fisocd{\overline{X'}}_{\Qpbar}$. Then $f^*L_1$ would also extend to $\overline{X'}$ by \cite[Corollary 3.3.3]{kedlayacompanions}, contradicting the assumption that $L_1$ had infinite monodromy at $\infty$. Moreover, each $\calE_i$ is irreducible by \cite[Lemma 3.3.1]{kedlayacompanions}. Similarly, tensorial operations respect the companion relation, hence $\det(\calE_i)\cong \Qpbar(-1)$. There exists a $p$-adic local field $K$ with each $\calE_i$ an object of $\fisocd{X}_{K}$. Set $\calE:=\bigoplus_{i=1}^b
\calE_i$, considered as an object of $\fisocd{X}$ (by restricting scalars from $K$ to $\Qp$, so the rank of $\calE$ is $2b[K:\QQ]$). Note that $\calE$, being the sum of irreducible objects, is semisimple. Then $\calE$ satisfies the hypotheses of Theorem \ref{main_thm}, and moreover $L_1$ is a companion of a summand of $\calE$. It follows that there is an open set $U\subset X$ together with an abelian scheme $\pi_U\colon A_U\rightarrow U$ such that $\calE\cong \mathbb{D}(A_U[p^{\infty})\otimes \Qp$. Again using Zarhin's semi-simplicity, $L_1|_U$ is a summand of $R^1(\pi_U)_*\Qlbar$, as desired. \end{proof}

\begin{proof}[Proof of Corollary \ref{cor:higher_dimensional_padic_drinfeld}]
Under the assumption on $E_1$, all $p$-adic companions to $\calE_1$ exist by \cite[Corollary 4.16]{krishnamoorthypal2018}. (This result is straightforward; they are all Galois twists of each other.) Fix $\sigma\colon \Qpbar \rightarrow \Qlbar$. Then the $\sigma$-companion to $\calE_1$ exists by \cite[Theorem 4.2]{abe2016lefschetz} or \cite[Corollary 3.5.3]{kedlayacompanions}. Apply Corollary \ref{cor:higher_dimension_drinfeld}.
\end{proof}

\appendix
\section{Logarithmic $F$-crystals}\label{section:log}
We first recall the notion of a logarithmic $F$-crystal/isocrystal. While this notion is due to Kato \cite[Section 6]{kato1989log}, our treatment is copied from recent work of Kedlaya. 
\begin{defn}A \emph{smooth pair} over a perfect field $k$ is a pair $(Y,Z)$ where $Y/k$ is a smooth variety and $Z\subset Y$ is a strict normal crossings divisor.
\end{defn}
\begin{defn}Let $(Y,Z)$ be a smooth pair over a perfect field $k$ of characteristic $p>0$. A \emph{smooth chart} for $(Y,Z)$ is a sequence of elements $\bar{t}_1,\dots,\bar{t}_n$ of elements of $\calO_Y(Y)$ such that the
\begin{itemize}
\item induced map $\bar{f}\colon Y\rightarrow \mathbb{A}^n$ is \'etale, and
\item there exists an $m\in [1,n]$ such that the zero-loci of $\bar{t_i}$, for $i=1\dots m$, are exactly the irreducible components of $Z$.
\end{itemize}
\end{defn}
Smooth charts exist Zariski locally on smooth pairs (in characteristic $p$) by \cite[Theorem 2]{kedlayaetaleaffine}.  Let $(Y,Z)$ be a smooth pair over a perfect field $k$ of characteristic $p>0$. Let $\bar{t}_1,\dots,\bar{t}_n$ be a smooth chart of $(Y,Z)$. Let $P_0$ be the formal scheme given by the formal completion of $W(k)[t_1,\dots,t_n]$ along $(p)$. By topological invariance of the \'etale site, there exists a unique smooth formal scheme $P$ together with an \'etale morphism $f\colon P\rightarrow P_0$ lifting $\bar{f}$. We call the pair $(P,t_1,\dots,t_n)$ the \emph{lifted smooth chart} of $(Y,Z)$ associated to the original chart.

Let $\sigma_0\colon P_0\rightarrow P_0$ be the Frobenius lift with $\sigma^*(t_i)=t_i^p$ for $i\in [1,\dots,n]$. Then there exists an associated Frobenius lift $\sigma\colon P\rightarrow P$.
\begin{defn}\label{def:log_f_crystal}Let $(Y,Z)$ be a smooth pair over a perfect field $k$ and let  $\bar{t}_1,\dots,\bar{t}_n$ be a smooth chart of $(Y,Z)$. Keep notations as above. A \emph{logarithmic crystal with nilpotent residues} on $(Y,Z)$ is a pair $(M,\nabla)$ where 
\begin{itemize}
\item $M$ is a $p$-torsion free coherent module over $P$; and
\item $\nabla$ is an integrable, topologically quasi-nilpotent connection on $M$ (with respect to $W(k)$) with logarithmic poles and nilpotent residues along the zero-loci of $f^*(t_i)$ for $i\in 1,\dots,m$.
\end{itemize} A \emph{logarithmic $F$-crystal with nilpotent residues} is a triple $(M,\nabla,F)$ where $(M,\nabla)$ is a logarithmic crystal with nilpotent residues and $F$ is an injective, horizontal morphism  $$F\colon \sigma^*(M)\rightarrow M$$ of coherent $P$-modules. A \emph{logarithmic Dieudonn\'e crystal with nilpotent residues} is a quadruple $(M,\nabla,F,V)$ where $(M,\nabla,F)$ is a logarithmic $F$-crystal in finite, locally free modules with nilpotent residues and $V$ is an injective, horizontal map
$$V\colon (M)\rightarrow \sigma^* M$$
such that $FV=VF=p$.
\end{defn}
\begin{rem}In the definition of a logarithmic $F$-crystal with nilpotent residues, we do not demand that $M$ is locally free. However, in our definition of a logarithmic Dieudonn\'e crystal, we do demand that the underlying logarithmic crystal is locally free.
\end{rem}

This definition extends to general smooth pairs by Zariski gluing; every smooth pair admits a finite open covering on which the restriction admits a smooth chart. We often drop the connection $\nabla$ from the notation and write a logarithmic $F$-crystal as $(M,F)$.

There is a natural category of logarithmic crystals with nilpotent residues on $(Y,Z)$ (where morphisms are $P$-linear and horizontal), and the category of logarithmic isocrystals with nilpotent residues is defined to be the induced isogeny category. One similarly defines the category of  logarithmic $F$-isocrystals with nilpotent residues.

\begin{rem}\label{rem:nilpotent}Part of the definition of a logarithmic $F$-crystal $(M,\nabla,F)$ in Definition \ref{def:log_f_crystal} explicitly assumes that the residues of the underlying crystal $(M,\nabla)$ were nilpotent. This assumption is indeed superfluous; we now exlain why.

First of all, the associated logarithmic isocrystal to $(M,\nabla)\otimes \QQ$ to $(M,\nabla,F)$ is a convergent logarithmic isocrystal: indeed, a logarithmic isocrystal is convergent if and only if it is infinitely Frobenius divisible, see \cite[Remark 16]{ogus1995constant}, the argument of which is just a logarithmic variant of \cite[2.18]{ogus1984f}. (See \cite[Section 2.4]{berthelot1996cohomologie} or \cite[Remark 2.4]{esnaultshiho2018} for several other perspectives in the non-logarithmic setting.) Then it is a general fact that a convergent logarithmic $F$-isocrystal has nilpotent residues, see e.g. \cite[Definition 7.2]{kedlaya2016notes}.
\end{rem}
\begin{rem}Let $(Y,Z)$ be a smooth pair over $k$ and let $U=Y\backslash Z$. We denote by $\underline{Y}$ the (fine, saturated) logarithmic scheme given by $(Y,\alpha\colon \calO_U^*\hookrightarrow \calO_Y)$. Then our definition of a logarithmic crystal is compatible with the definition of Kato (see \cite[Theorem 6.2]{kato1989log}), our definition of a logarithmic $F$-crystal in finite, locally free modules is compatible with the definition of Kato-Trihan (see \cite[4.1]{katotrihan}) and our definition of a logarithmic $F$-isocrystal is compatible with the definition given by Shiho (see \cite[Definition 4.1.3]{shiho2000crystalline}).
\end{rem}
The mathematical content of the following lemma is essentially \cite[Theorem 2.6.1]{katz1979slope} (and relatedly \cite[Lemma 2.5.1]{crew1987bowdoin}); we have simply rewritten Katz's argument to the logarithmic setting. The key is that Katz's slope bounds holding on the open subset where the logarithmic structure is trivial guarantee that they hold everywhere. We use Kato's definition of logarithmic $F$-crystals only for convenience to discuss global objects; all of the computations use the local definitions given above.
\begin{lem}\label{lem:katz}Let $(Y,Z)$ be a smooth pair over a perfect field $k$ of positive characteristic and let $U:=Y\backslash Z$. Let $\calE$ be a logarithmic $F$-isocrystal on $(Y,Z)$.
\begin{enumerate}
\item Suppose the Newton slopes of $\calE_U$ are all non-negative. Then there exists an open subset $W\subset Y$, whose complementary codimension is at least $2$, and a logarithmic $F$-crystal in finite, locally free modules $(M'',F)$ on the smooth pair $(W,W\cap Z)$ such that $(M'',F)\otimes \QQ\cong \calE_W$.
\item Suppose the Newton slopes of $\calE_U$ are in the interval $[0,1]$. Then there exists an open subset $W\subset Y$, whose complementary codimension is at least $2$, and a logarithmic Dieudonn\'e crystal in finite, locally free modules $(M'',F,V)$ on the smooth pair $(W,W\cap Z)$ such that $(M'',F)\otimes \QQ\cong \calE_W$.
\end{enumerate}
\end{lem}
\begin{proof}By definition of a logarithmic $F$-isocrystal, there exists a logarithmic crystal in coherent (not necessarily locally free!) modules $M$ and a map $F\colon Frob_{\underline Y}^* M\rightarrow M\otimes \QQ$ that is isomorphic to $\calE$ when thought of as a logarithmic $F$-isocrystal. Here, $Frob_{\underline Y}$ refers to the absolute Frobenius (on the f.s. log scheme $\underline{Y}$ induced from the smooth pair $(Y,Z)$) and the $*$ refers to pullback on the logarithmic crystalline topos. This is compatible with our above definitions.

As $M$ is a logarithmic crystal in \emph{coherent} modules, there exists a non-negative integer $\nu$ so that

$$F\colon (Frob_{\underline Y})^*M\rightarrow p^{-\mu}M.$$
We have assumed that the Newton slopes of $\calE$ are all non-negative. Slope bounds of Katz (see the proof on \cite[p. 151-152]{katz1979slope}) imply then that there exists a non-negative $\nu$ such that all $n\geq 0$,
\begin{equation}\label{equation:katz_slope_bounds}\displaystyle F^n\colon (Frob^n_U)^*M_U\rightarrow p^{-\nu}M_U.
\end{equation}

We explicate this in local coordinates. Take an affine open neighborhood $T\subset Y$ such that $(T,T\cap Z)$ has a smooth chart $(\bar{t}_1,\dots,\bar{t}_n)$. Let $(P,t_1,\dots,t_n)$ be the associated lifted smooth chart; note that $P=\text{Spf}(A)$ where $A$ is a noetherian $W(k)$ algebra equipped with the $p$-adic topology. Then the logarithmic crystal yields a finitely generated $A$ module $M_A$ and the Frobenius structure induces a continuous, $A$-linear homomorphism $F\colon \sigma^*M_A\rightarrow p^{-\mu}M_A$. 


As $U\cap T\subset T$ is open dense, it follows from Equation \ref{equation:katz_slope_bounds} that $$\displaystyle F^n \colon (\sigma^n)^*M_A\rightarrow p^{-\nu}M_A.$$. By varying $T$, one deduces that $F^n\colon (Frob^n_{\underline Y})^*M\rightarrow p^{-\nu}M$ for our fixed $\nu$ as above and for all $n\geq 0$.

Consider the module
$$\displaystyle M_A':=\sum_{n\geq 0}F^n\left((\sigma^n_X)^*M_A\right)\subset p^{-\nu}M_A.  $$
As $A$ is noetherian, $M_A'$ is finitely generated, being a submodule of a finitely generated module. Moreover, $M_A'$ is stable under $F$. Finally, $M_A'$ is the finite sum of (logarithmic) horizontal submodules. Therefore the pair $(M_A',F)$ is in fact a logarithmic $F$-crystal in coherent modules. We have an isomorphism $(M_A',F)\otimes \QQ\cong \calE_T$ in the category of logarithmic $F$-isocrystals with nilpotent residues on $(T,Z\cap T)$.
 
Now, set $M_A'':=(M_A')^{**}$. This is a coherent reflexive sheaf on the ring $A$, and hence is locally free away on an open set of $\Spec(A)$ whose complement has codimension at least 3 \cite[0AY6]{stacks-project}. $M_A''$ is manifestly stable under the connection and $F$. In particular, we can find an open subset $T''\subset T$ with complementary codimension at least 2 such that the logarithmic $F$-crystal $(M_A'',F)_{T''}$ is a crystal in finite, locally free modules.

After initially choosing a pair $(M,F\colon Frob_{\underline{Y}}^*M\rightarrow p^{-\mu}M)$ representing $\calE$, the constructions we have made are canonical. Therefore, ranging over $T$, we may glue the $(M'',F)_{T''}$; that is, there is an open subset $W\subset T$ with complementary codimension at least 2 and a logarithmic $F$-crystal $(M'',F)_W$ in finite, locally free modules on the smooth pair $(W,Z\cap W)$ that is a lattice inside of $\calE_W$. 

We now indicate how to complete the result if the Newton polygons on $U$ are in the interval $[0,1]$. Let $(M,F)$ be a logarithmic $F$-crystal in finite, locally free modules on a smooth pair $(Y,Z)$ over a perfect field $k$ and suppose the Newton slopes on $U$ are no greater than 1. Set $V:=F^{-1}\circ p$. Then $V$ \emph{does not} necessarily stabilize $M$; however, the pair $(M,V)_U$ is a logarithmic $\sigma^{-1}$-F-isocrystal in the language of \cite{katz1979slope}. (Fortunately, Katz's entire paper is written in the context of $\sigma^a$-$F$-crystals for \emph{any} $a\neq 0$, not just the positive $a$. In particular, all of Katz's results also hold for $\sigma^{-1}$-$F$-crystals. Katz does not deal with logarithmic crystals, but we only use the slope bounds on the open set $U$.) By the coherence argument as above, we may find $\eta$ such that:
$$V\colon (Frob_{\underline Y}^{-1})^*M\rightarrow p^{-\eta}M$$
on all of $Y$. Again, using Katz's slope bounds on $U$ (which hold equally well for $\sigma^{-1}$-$F$-crystals) and the same coherence argument, one shows that after possibly increasing $\eta$, we in fact have
$$V^n\colon (Frob_{\underline Y}^{-n})^*M\rightarrow p^{-\eta}M$$
for all $n\geq 0$.
 Now run exactly the above argument with $V$ instead of $F$: then
 $$\displaystyle M':=\sum_{n\geq 0} V^n (Frob_{\underline Y}^{-1})^*M)\subset p^{-\eta} M$$
 will be coherent, horizontal, and stabilized by $V$. Recall that $FV=VF=p$; therefore $M'$ is also stabilized by $F$! Then $M'':=(M')^{**}$ is a reflexive logarithmic crystal on $(Y,Z)$ that is stabilized by both $F$ and $V$. Exactly as above, there exists an open subset $W\subset Y$ of complementary codimension at least 2 such that $(M'',F,V)_W$ is a logarithmic Dieudonn\'e crystal in finite, locally free modules, as desired. 

\end{proof}
\begin{rem}\label{rmk:hodge}
Let $(Y,Z)$ be a smooth pair over $k$ and let $(M,F,V)$ be a logarithmic Dieudonn\'e crystal (in finite, locally free modules) on $(Y,Z)$. We construct a natural line bundle $\omega$, which we call the \emph{Hodge line bundle}, attached to $(M,F,V)$.

Evaluating $M$ on the trivial thickening of $(Y,Z)$, we obtain a vector bundle $M_{(Y,Z)}$ on $Y$ together with an integrable connection with logarithmic poles on $Z$ and a horizontal map:

$$F_{(Y,Z)}\colon Frob^*_{\underline{Y}} M_{(Y,Z)} \rightarrow M_{(Y,Z)}.$$
The kernel is a vector bundle on $Y$. Set $\omega:=\det(\text{ker}(F_{(Y,Z)}))$. We call $\omega$ the \emph{Hodge line bundle} associated to $(M,F)$.

As a reference for this remark: in the case when $Z$ is empty, one finds this construction in \cite[2.5.2 and 2.5.5]{de1998homomorphisms}. In the setting of logarithmic Dieudonn\'e crystals, Kato-Trihan construct the dual object: $Lie(M,F,V)$, see \cite[5.1]{katotrihan} and especially Lemma 5.3 of \emph{loc. cit.} Note that this lemma holds in our situation: our hypothesis that $(Y,Z)$ is a smooth pair over a perfect field $k$ implies that the conditions of 5.1 on p. 563 of \emph{loc. cit.} hold: \'etale locally, there is a $p$-basis of $Y$ such that each (regular) component of $Z$ is cut out by some member of the $p$-basis. 
\end{rem}
\begin{rem}\label{rem:Hodge_compatible}
Let $Y/k$ be a smooth scheme over a perfect field $k$. Let $A_Y\rightarrow Y$ be an abelian scheme. Then there is an associated Dieudonn\'e crystal $(M,F,V)=\mathbb{D}(A_Y[p^{\infty}])$ on $Y$ \cite{berthelot1982theorie}. The Hodge bundle of $(M,F)$ is isomorphic to the Hodge line bundle of the abelian scheme $A_Y\rightarrow Y$ by \cite[3.3.5 and 4.3.10]{berthelot1982theorie}. 
\end{rem}

Finally, we have the following key Proposition \ref{prop:log_Hodge_compatible}, which furnishes several compatibilities that we need for our main argument. The proof of the proposition largely amounts to collating well-known results in the theory of $F$-(iso)crystals. We first require the following setup. 
\begin{setup}\label{setup:log_Hodge_compatible}Let $C/k$ be a smooth, proper, geometrically irreducible curve over a perfect field $k$ of characteristic $p>0$, let $U\subset C$ be an open dense subset, and let $Z\subset C$ be the reduced complement. Let $A_U\rightarrow U$ be an abelian scheme with semi-stable reduction along $Z$. Call the N\'eron model $A_C\rightarrow C$. Then there is an attached logarithmic Dieudonn\'e crystal on $(C,Z)$, which we call $\mathbb{D}^{\log}(A_C)$ \cite[4.4-4.8]{katotrihan}. (Kato-Trihan construct a covariant Dieudonn\'e functor. We assume ours is contravariant, which may be accomplished by taking a dual as in \cite[4.1]{katotrihan}.)
\end{setup}
\begin{prop}\label{prop:log_Hodge_compatible}
In the context of Setup \ref{setup:log_Hodge_compatible}, the following hold.
\begin{enumerate}
\item The (non-logarithmic) Dieudonn\'e crystal $\mathbb{D}^{\log}(A_C)|_U$ is isomorphic to the crystalline Dieudonn\'e module of the $p$-divisible group $A_U[p^{\infty}]$
\item Let $A^o_C\rightarrow C$ be the associated semi-abelian scheme to $A_C\rightarrow C$, obtained by removing the non-identity components of the fibers over $Z$. The Hodge line bundle of $A^o_C\rightarrow C$, is isomorphic to the Hodge line bundle of the logarithmic Dieudonn\'e crystal $\mathbb{D}^{\log}(A_C)$ described in Remark \ref{rmk:hodge}.
\item The logarithmic Dieudonn\'e crystal $\mathbb{D}^{\log}(A_C)$ is the \textbf{unique} logarithmic Dieudonn\'e crystal (with nilpotent residues) on $(C,Z)$ that extends $\mathbb{D}(A_U[p^{\infty}])$.
\end{enumerate}
\end{prop}
\begin{proof}
We prove each point in turn. For the first point, this follows by the construction of the logarithmic Dieudonn\'e module: see the description of gluing as in \cite[Lemma 4.4.1]{katotrihan}.

The second point is given in \cite[Example 5.4(b)]{katotrihan}, with the caveat that they work with the covariant Dieudonn\'e functor and $\text{Lie}(A_C\rightarrow C)$. 

We now prove the final point. First of all, note that we only need to check that there is at most one extension as a logarithmic $F$-crystal in finite, locally free modules: in our setting, $V$ is determined by $F$ under the relation $FV=VF=p$. By \cite[Th\'eor\`eme 7]{etesse}, it follows that $\mathbb{D}(A_U[p^{\infty}])\otimes \Qp$ is overconvergent. Forgeting the $V$-structure, we are left with a logarithmic $F$-crystal $(M,F)$ on $(C,Z)$. (By Remark \ref{rem:nilpotent}, the residues of $(M,F)$ are automatically nilpotent.) Note that $(M,F)|_U$ is an overconvergent $F$-crystal by \cite{kedlaya2001full}.



Now we are able to prove the desired uniqueness. Let $(N,F)$ be a logarithmic $F$-crystal on $(C,Z)$ such that $(M,F)|_U\cong (N,F)|_U$. (By the above, $(N,F)$ automatically has nilpotent residues along $Z$.) We introduce the following notation.
\begin{itemize}
\item $\fc{C,Z}$ is the category of logarithmic $F$-crystals in finite, locally free modules (with nilpotent residues) on $(C,Z)$.
\item $\fc{U}$ is the category of $F$-crystals in finite, locally free modules on $U$.
\item $\fisoc{C,Z}$ is the category of logarithmic $F$-isocrystals (with nilpotent residues) on $(C,Z)$.
\item $\fisoc{U}$ is the category of (convergent) $F$-isocrystals (with nilpotent residues) on $U$.
\end{itemize}
Consider the following diagram.

\begin{equation}\label{eqn:uniqueness}
\xymatrix{
\Hom_{\fc{C,Z}}\big((M,F),(N,F)\big)\ar[r]^{\text{res}}\ar[d]	& \Hom_{\fc{U}}\big((M,F)|_U,(N,F)|_U\big)\ar[d]\\
\Hom_{\fisoc{C,Z}}\big((M,F)\otimes \QQ,(N,F)\otimes \QQ\big)\ar[r]^{\text{res}_{\QQ}}	& \Hom_{\fisocd{U}}\big((M,F)|_U\otimes \QQ,(N,F)_{U}\otimes \QQ)
}\end{equation}
To prove that $(M,F)\cong (N,F)$ in the category $\fc{C,Z}$, it suffices to show that the top horizontal arrow is an isomorphism. We first prove that this arrow is injective with torsion cokernel; then we will show that the image is $p$-saturated.

The natural map

$$\Hom_{\textbf{Isoc}(C,Z)}\big( M\otimes \QQ,N\otimes \QQ\big)\rightarrow \Hom_{\textbf{Isoc}^{\dagger}(U)}\big( M|_U\otimes \QQ,N|_U\otimes \QQ\big)$$ is an isomorphism by a full-faithfulness result of Kedlaya \cite[Theorem 6.4.5]{kedlayasemistableI}. It follows immediately that the bottom horizontal arrow of Equation \ref{eqn:uniqueness} is an isomorphism. (See also \cite[Theorem 7.3]{kedlaya2016notes} for exactly this statement.)

The group $\Hom_{\fc{C,Z}}\big((M,F),(N,F)\big)$ is a finite free $\Zp$-module because $(C,Z)/k$ is log smooth. The left vertical arrow of Equation \ref{eqn:uniqueness} is injective because it is simply the map $\otimes \QQ$ on a finite free $\Zp$-module.

The group $\Hom_{\fc{U}}\big((M,F)|_U,(N,F)|_U\big)$ is \`a priori only a torsion-free $\Zp$-module. (In particular, it could have infinite rank.) However, the right vertical arrow fits into the following diagram:

$$\xymatrix{
\Hom_{\fc{U}}\big((M,F)|_U,(N,F)|_U\big)\ar[r]\ar[dr]_{\otimes \QQ}	&
\Hom_{\fisocd{U}}\big((M,F)|_U\otimes \QQ,(N,F)|_U\otimes \QQ\big)\ar[d]^{\cong}\\
&	\Hom_{\fisoc{U}}\big((M,F)|_U\otimes \QQ,(N,F)|_U\otimes\QQ\big)
}$$
where the vertical arrow is an isomorphism by \cite{kedlaya2001full}. As 
$$\Hom_{\fisocd{U}}\big((M,F)|_U\otimes \QQ,(N,F)|_U\otimes\QQ\big)$$
is a finite-dimensional $\Qp$-vector space, it follows that $\Hom_{\fc{U}}\big((M,F)|_U,(N,F)|_U\big)$ is a finite free $\Zp$-module. As the diagonal arrow is injective (it is the map $\otimes \QQ$), we deduce that:

$$\Hom_{\fc{U}}\big((M,F)|_U,(N,F)|_U\big)\rightarrow \Hom_{\fisocd{U}}\big((M,F)|_U\otimes \QQ,(N,F)|_U\otimes\QQ\big)$$
is also injective. Then the top arrow, $\text{res}$, in Diagram \ref{eqn:uniqueness} must be injective with torsion cokernel by consideration of the ranks. 

Finally, we prove that $\text{res}$ is a saturated map of finite free $\Zp$ modules. Equivalently, we prove that if $\varphi\in \Hom_{\fc{C,Z}}\big((M,F),(N,F)\big)$ is such that $\text{res}(\varphi)$ is divisible by $p$ in $\Hom_{\fc{U}}\big((M,F)|_U,(N,F)|_U\big)$, then $\varphi$ is divisible by $p$. This will use an explicit local calculation with Definition \ref{def:log_f_crystal}.

Pick smooth charts for $(C,Z)$. More precisely, $C$ may be covered by open subsets $Y_i$ such that there exists \'etale maps $\bar{f}_i\colon Y_i\rightarrow \mathbb{A}^1=\Spec(k[x])$ with the following property: if $Y_i\cap Z\neq \emptyset$, then $Y_i\cap Z=\bar{f}_i^*(V(x))$. As both the category of logarithmic $F$-crystals and usual $F$-crystals are stacks in the Zariski topology, it suffices to prove the desired saturatedness for a single $(Y_i,Y_i\cap Z)$, which we relable $(Y,Y\cap Z)$. If $Y\cap Z=\emptyset$, there is nothing to prove, so we may assume that $Y\cap Z\neq \emptyset$. 

Let $\bar{f}\colon Y\rightarrow \mathbb A^1$ be an \'etale map with $\bar{f}^*(V(x))=Y\cap Z$, which exists by the definition of $Y$. Let $A_0:=W(k)[x]^{\wedge}$ be the $p$-adic completion of $W(k)[x]$. By topological invariance of the \'etale site, the map $k[x]\rightarrow \mathcal O_Y(Y)$ (with $x\mapsto \bar f$) deforms to an \'etale map $A_0\rightarrow A$; set $f$ to be the image of $x$ in $A$. Similarly, let $B_0=W(k)[x,x^{-1}]^{\wedge}$ be the $p$-adic completion of $W(k)[x,x^{-1}]$. Again using topological invariance of the \'etale site, the map $k[x,x^{-1}]\rightarrow \mathcal{O}_Y(Y\setminus (Y\cap Z))$ induced from $\bar f$ deforms to an \'etale map  $B_0\rightarrow B$. As $Y\cap Z=\bar{f}^*(V(x))$, we have that $B\cong A\hat{\otimes}_{A_0} B_0$. In particular, there is the following diagram of $p$-adic rings.

$$\xymatrix{
A\ar@{^{(}->}[r]	& B\\
A_0\ar[u]\ar@{^{(}->}[r]	& B_0.\ar[u]
}
$$

\noindent As in Definition \ref{def:log_f_crystal}, equip $A_0$ and $B_0$ with the Frobenius lift $\sigma_0$ sending $t\mapsto t^p$. Set $\sigma$ to be the induced Frobenius lift on $A$ and $B$.

Let $(M,\nabla,F)$ and $(N,\nabla,F)$ be the realizations of $(M,F)$ and $(N,F)$ on $A$ as in Definition \ref{def:log_f_crystal}. In particular, $M$ and $N$ are finite, locally free $A$ modules. Then the statement we wish to prove is the following:

$$\Hom_A\big((M,\nabla,F),(N,\nabla,F)\big)\rightarrow \Hom_B\big((M,\nabla,F)_B,(N,\nabla,F)_B\big)$$
is a saturated map of $\Zp$ modules, i.e., if $\varphi\in \Hom_A\big((M,\nabla,F),(N,\nabla,F)\big)$ is a map such that $\varphi_B$ is divisible by $p$ in $\Hom_B\big((M,\nabla,F)_B,(N,\nabla,F)_B\big)$, then $\varphi$ was divisible by $p$. In particular, we assume that $\varphi(M_B)\subset pN_B$ and wish to prove that $\varphi(M)\subset pN$; indeed, if $\varphi(M)\subset pN$, then $\frac{\varphi}{p}$ will automatically commute with $\nabla$ and $F$ and hence would yield an element $\frac{\varphi}{p}\in\Hom_A\big((M,\nabla,F),(N,\nabla,F)\big)$. Therefore, it suffices to prove that if $M$ and $N$ are finite, locally free $A$ modules, then the map $\Hom_A(M,N)\rightarrow \Hom_B(M_B,N_B)$ is $p$-saturated.

We claim that $A\hookrightarrow B$ is a $p$-saturated map of $p$-adic rings. As noted above, $B\cong A\hat{\otimes}_{A_0} B_0$; therefore, to prove that $A\hookrightarrow B$ is a $p$-saturated, it suffices to prove that $A_0\hookrightarrow B_0$ is $p$-saturated. This map is simply the inclusion $W(k)[x]^{\wedge}\hookrightarrow W(k)[x,x^{-1}]^{\wedge}$, which is clearly $p$-saturated from the explicit description of the elements of the two rings as series.

 As $M$ and $N$ are finite locally free $A$-modules, the natural map $\Hom_A(M,N)\otimes_A B\rightarrow \Hom_B(M_B,N_B)$ is an isomorphism. It follows that the natural map:
$$\Hom_A(M,N)\rightarrow \Hom_A(M,N)\otimes_A B\cong \Hom_B(M_B,N_B)$$
is $p$-saturated, as desired.

\end{proof}

\begin{acknowledgement*}
This work was born at CIRM (in Luminy) at ``$p$-adic Analytic Geometry and Differential Equations''; the authors thank the organizers. R.K. warmly thanks Valery Alexeev, Philip Engel, Kiran Kedlaya, Daniel Litt, and especially Johan de Jong, with whom he had stimulating discussions on the topic of this article. In addition, the authors heartily thank the anonymous referee for a detailed, thorough, and helpful report. R.K. gratefully acknowledges financial support from the NSF under Grants No. DMS-1605825 and No. DMS-1344994 (RTG in Algebra, Algebraic Geometry and Number Theory at the University of Georgia).
\end{acknowledgement*}

\bibliographystyle{plain}
\bibliography{extending_abelian}

\begin{thebibliography}{10}

\bibitem{abe2013langlands}
Tomoyuki Abe.
\newblock Langlands correspondence for isocrystals and the existence of
  crystalline companions for curves.
\newblock {\em J. Amer. Math. Soc.}, 31(4):921--1057, 2018.

\bibitem{abe2016lefschetz}
Tomoyuki Abe and H{\'{e}}l{\`{e}}ne Esnault.
\newblock A {L}efschetz theorem for overconvergent isocrystals with frobenius
  structure.
\newblock {\em Ann. Scient. {\'{E}}co. Norm. Sup.}, 52(5):1243--1264, 2019.

\bibitem{berthelot1996cohomologie}
Pierre Berthelot.
\newblock Cohomologie rigide et cohomologie rigidea supports propres, premiere
  partie, pr{\'e}publication irmar, 1996.

\bibitem{berthelot1997finitude}
Pierre {Berthelot}.
\newblock {Finitude et puret\'e cohomologique en cohomologie rigide (avec un
  appendice par Aise Johan de Jong)}.
\newblock {\em {Invent. Math.}}, 128(2):329--377, 1997.

\bibitem{berthelot1982theorie}
Pierre Berthelot, Lawrence Breen, and William Messing.
\newblock {\em Th{\'e}orie de Dieudonn{\'e} cristalline II}, volume 930 of {\em
  Lecture Notes in Mathematics}.
\newblock Springer-Verlag, Berlin, 1982.

\bibitem{chai2014complex}
Ching-Li Chai, Brian Conrad, and Frans Oort.
\newblock {\em Complex multiplication and lifting problems}, volume 195 of {\em
  Mathematical Surveys and Monographs}.
\newblock American Mathematical Society, Providence, RI, 2014.

\bibitem{corlette2008classification}
Kevin Corlette and Carlos Simpson.
\newblock On the classification of rank-two representations of quasiprojective
  fundamental groups.
\newblock {\em Compos. Math.}, 144(5):1271--1331, 2008.

\bibitem{crew1987bowdoin}
Richard {Crew}.
\newblock {F-iscocrystals and p-adic representations.}
\newblock {Algebraic geometry, Proc. Summer Res. Inst., Brunswick/Maine 1985,
  part 2, Proc. Symp. Pure Math. 46, 111-138 (1987).}, 1987.

\bibitem{de1995crystalline}
A.~J. de~Jong.
\newblock Crystalline {D}ieudonn\'e module theory via formal and rigid
  geometry.
\newblock {\em Inst. Hautes \'Etudes Sci. Publ. Math.}, (82):5--96 (1996),
  1995.

\bibitem{de1998homomorphisms}
A.~J. de~Jong.
\newblock Homomorphisms of {B}arsotti-{T}ate groups and crystals in positive
  characteristic.
\newblock {\em Invent. Math.}, 134(2):301--333, 1998.

\bibitem{deligne1980conjecture}
Pierre Deligne.
\newblock La conjecture de {W}eil. {II}.
\newblock {\em Inst. Hautes \'Etudes Sci. Publ. Math.}, (52):137--252, 1980.

\bibitem{deligne2012finitude}
Pierre Deligne.
\newblock Finitude de l'extension de {$\Bbb Q$} engendr\'ee par des traces de
  {F}robenius, en caract\'eristique finie.
\newblock {\em Mosc. Math. J.}, 12(3):497--514, 668, 2012.

\bibitem{drinfeld1977}
V.~G. {Drinfel'd}.
\newblock {Elliptic modules. II.}
\newblock {\em {Math. USSR, Sb.}}, 31:159--170, 1977.

\bibitem{drinfeld1983}
V.~G. {Drinfeld}.
\newblock {Two-dimensional \(\ell\)-adic representations of the fundamental
  group of a curve over a finite field and automorphic forms on
  \(\mathrm{GL}(2)\).}
\newblock {\em {Am. J. Math.}}, 105:85--114, 1983.

\bibitem{drinfeld2012conjecture}
Vladimir Drinfeld.
\newblock {On a conjecture of Deligne}.
\newblock {\em Moscow Mathematical Journal}, 12(3):515--542, 2012.

\bibitem{drinfeld2018pro}
Vladimir Drinfeld.
\newblock On the pro-semisimple completion of the fundamental group of a smooth
  variety over a finite field.
\newblock {\em Advances in Mathematics}, 327:708--788, March 2018.

\bibitem{drinfeld2016slopes}
Vladimir Drinfeld and Kiran~S. Kedlaya.
\newblock Slopes of indecomposable {$F$}-isocrystals.
\newblock {\em Pure and Applied Mathematics Quarterly}, 13(1):131--192, 2017.

\bibitem{esnaultkindler}
H\'el\`ene {Esnault} and Lars {Kindler}.
\newblock {Lefschetz theorems for tamely ramified coverings}.
\newblock {\em {Proc. Am. Math. Soc.}}, 144(12):5071--5080, 2016.

\bibitem{esnaultshiho2018}
H{\'e}l{\`e}ne Esnault and Atsushi Shiho.
\newblock Convergent isocrystals on simply connected varieties.
\newblock {\em {Ann. Inst. Fourier}}, 68(5):2109--2148, 2018.

\bibitem{etesse}
{Jean-Yves} {\'E}tesse.
\newblock {Descente \'etale des \(F\)-isocristaux surconvergents et
  rationalit\'e des fonctions \(L\) de sch\'emas ab\'eliens}.
\newblock {\em {Ann. Sci. \'Ec. Norm. Sup\'er. (4)}}, 35(4):575--603, 2002.

\bibitem{faltings2013degeneration}
Gerd Faltings and Ching-Li Chai.
\newblock {\em Degeneration of abelian varieties}, volume~22 of {\em Ergebnisse
  der Mathematik und ihrer Grenzgebiete (3) [Results in Mathematics and Related
  Areas (3)]}.
\newblock Springer-Verlag, Berlin, 1990.
\newblock With an appendix by David Mumford.

\bibitem{hartshorneample}
Robin Hartshorne.
\newblock {\em Ample subvarieties of algebraic varieties}.
\newblock Lecture Notes in Mathematics, Vol. 156. Springer-Verlag, Berlin-New
  York, 1970.
\newblock Notes written in collaboration with C. Musili.

\bibitem{kato1989log}
Kazuya Kato.
\newblock Logarithmic structures of {F}ontaine{-I}llusie.
\newblock In {\em {Algebraic analysis, geometry, and number theory: proceedings
  of the JAMI inaugural conference, held at Baltimore, MD, USA, May 16-19,
  1988}}, pages 191--224. Baltimore: Johns Hopkins University Press, 1989.

\bibitem{katotrihan}
Kazuya Kato and Fabian Trihan.
\newblock On the conjectures of {B}irch and {Swinnerton-Dyer} in characteristic
  {\(p>0\)}.
\newblock {\em {Invent. Math.}}, 153(3):537--592, 2003.

\bibitem{katz1979slope}
Nicholas~M. Katz.
\newblock Slope filtration of {$F$}-crystals.
\newblock In {\em Journ\'ees de {G}\'eom\'etrie {A}lg\'ebrique de {R}ennes
  ({R}ennes, 1978), {V}ol. {I}}, volume~63 of {\em Ast\'erisque}, pages
  113--163. Soc. Math. France, Paris, 1979.
\newblock Available at
  \url{https://web.math.princeton.edu/~nmk/old/f-crystals.pdf}.

\bibitem{katz2001spacefilling}
Nicholas~M. Katz.
\newblock Space filling curves over finite fields.
\newblock {\em {Math. Res. Lett.}}, 6(5-6):613--624, 2001.

\bibitem{katzmessing}
Nicholas~M. Katz and William Messing.
\newblock {Some consequences of the Riemann hypothesis for varieties over
  finite fields.}
\newblock {\em {Invent. Math.}}, 23:73--77, 1974.

\bibitem{kedlaya2001full}
Kiran~S. Kedlaya.
\newblock Full faithfulness for overconvergent {$F$}-isocrystals.
\newblock In {\em Geometric aspects of {D}work theory. {V}ol. {I}, {II}}, pages
  819--835. Walter de Gruyter, Berlin, 2004.

\bibitem{kedlayaetaleaffine}
Kiran~S. Kedlaya.
\newblock {More \'etale covers of affine spaces in positive characteristic}.
\newblock {\em {J. Algebr. Geom.}}, 14(1):187--192, 2005.

\bibitem{kedlayasemistableI}
Kiran~S. Kedlaya.
\newblock {Semistable reduction for overconvergent \(F\)-isocrystals. I:
  Unipotence and logarithmic extensions.}
\newblock {\em {Compos. Math.}}, 143(5):1164--1212, 2007.

\bibitem{kedlaya2016notes}
Kiran~S. Kedlaya.
\newblock Notes on isocrystals.
\newblock {\em arXiv preprint arXiv:1606.01321v5}, 2016.

\bibitem{kedlayacompanions}
Kiran~S. Kedlaya.
\newblock {\'Etale and crystalline companions I}.
\newblock {\em arXiv preprint arXiv:1811.00204}, 2018.

\bibitem{kedlayacompanionsii}
Kiran~S. Kedlaya.
\newblock {\'Etale and crystalline companions II}.
\newblock {\em arXiv preprint arXiv:2008.13053}, 2021.

\bibitem{krishnamoorthypal2018}
Raju Krishnamoorthy and Ambrus P{\'a}l.
\newblock {Rank 2 local systems and Abelian varieties}.
\newblock {\em Sel. Math. New Ser.}, 27(4):1--40, 2021.

\bibitem{lafforgue2002chtoucas}
Laurent Lafforgue.
\newblock Chtoucas de {D}rinfeld et correspondance de {L}anglands.
\newblock {\em Invent. Math.}, 147(1):1--241, 2002.

\bibitem{moret1985pinceaux}
Laurent {Moret-Bailly}.
\newblock {\em {Pinceaux de vari\'et\'es ab\'eliennes. (Pencils of abelian
  varieties).}}, volume 129.
\newblock Soci\'et\'e Math\'ematique de France (SMF), Paris, 1985.

\bibitem{ogus1984f}
Arthur Ogus.
\newblock {$F$}-isocrystals and de {R}ham cohomology. {II}. {C}onvergent
  isocrystals.
\newblock {\em Duke Math. J.}, 51(4):765--850, 1984.

\bibitem{ogus1995constant}
Arthur {Ogus}.
\newblock {F-crystals on schemes with constant log structure}.
\newblock {\em {Compos. Math.}}, 97(1-2):187--225, 1995.

\bibitem{pal2015monodromy}
Ambrus P{\'a}l.
\newblock The $ p $-adic monodromy group of abelian varieties over global
  function fields of characteristic $ p$.
\newblock {\em arXiv preprint arXiv:1512.03587}, 2015.

\bibitem{poonen2004bertini}
Bjorn Poonen.
\newblock Bertini theorems over finite fields.
\newblock {\em Ann. of Math. (2)}, 160(3):1099--1127, 2004.

\bibitem{shiho2000crystalline}
Atsushi {Shiho}.
\newblock {Crystalline fundamental groups. I: Isocrystals on log crystalline
  site and log convergent site.}
\newblock {\em {J. Math. Sci., Tokyo}}, 7(4):509--656, 2000.

\bibitem{snowden2018constructing}
Andrew Snowden and Jacob Tsimerman.
\newblock Constructing elliptic curves from galois representations.
\newblock {\em Compositio Mathematica}, 154(10):2045--2054, 2018.

\bibitem{stacks-project}
The {Stacks Project Authors}.
\newblock \textit{Stacks Project}.
\newblock \url{https://stacks.math.columbia.edu}, 2020.

\bibitem{weilAV}
Andr\'e {Weil}.
\newblock {Vari\'et\'es ab\'eliennes et courbes alg\'ebriques}.
\newblock {Actualit\'es scientifiques et industrielles. 1064 = Publ. Math.
  Inst. Univ. Strasbourg. 8 (1946). Paris: Hermann \& Cie. 163 p. (1948).},
  1948.

\end{thebibliography}

\end{document}